\documentclass{amsart}
 \newtheorem{thm}{Theorem}[section]
 \newtheorem{cor}[thm]{Corollary}
 \newtheorem{lem}[thm]{Lemma}
 \newtheorem{prop}[thm]{Proposition}
 \theoremstyle{definition}
 \newtheorem{defn}[thm]{Definition}
 \theoremstyle{remark}
 \newtheorem{rem}[thm]{Remark}
 \newtheorem*{ex}{Example}
 \numberwithin{equation}{section}

\let\d\partial
\def\A{\mathbb A}

\def\H{\mathbb H}
\def\C{\mathbb C}
\def\Q{\mathbb Q}

\def\R{{\mathbb R}}
\def\Z{\mathbb Z}
\def\N{\mathbb N}

\def\CC{\mathcal C}
\def\DD{\mathcal D}
\def\FF{\mathcal F}
\def\EE{\mathcal E}
\def\HH{\mathcal H}
\def\LL{\mathcal L}

\def\MM{\mathcal M}
\def\NN{\mathcal N}
\def\OO{\mathcal O}
\def\PP{\mathcal P}

\def\WW{{\mathcal W}}
\def\VV{\mathcal V}

\def\n{\noindent}

\newcommand{\ke}{ \hbox{\rm Ker} }

\title {Variation of Mixed Hodge Structures}
\author{ Patrick Brosnan  }
\address{Department of Mathematics,The university of British  Columbia,
1984 Mathematics road, Vancouver, B.C. Canada V6T 1Z2}
\email{brosnan@math.ubc.ca}
\author{Fouad El Zein}
\address{Institut de Math\'ematiques de Jussieu.\\
Paris, France}
 \email{elzein@math.jussieu.fr}
 \keywords{Hodge theory, algebraic geometry, Mixed Hodge structure}
 \subjclass{Primary
 14D07, 32G20; Secondary 14F05}
\begin{document}
\begin{abstract}
 Variation of mixed Hodge structures(VMHS),
introduced by P. Deligne, is a linear structure reflecting the
geometry on cohomology of the fibers of an algebraic family,
generalizing variation of Hodge structures for smooth proper
families, introduced by P. Griffiths. Hence, it is a strong tool
to  study the variation of the geometric structure of fibers of a
morphism. We describe here the degenerating properties of a VMHS
of geometric origin and the existence of a relative monodromy
filtration, as well the definition and properties of abstract
admissible VMHS.
 \end{abstract}
 \maketitle
\section*{Introduction}
The object of the paper is to discuss the definition of admissible
 variations of
mixed Hodge structure (VMHS), the results in \cite{K} and
applications to the proof of algebraicity of the locus of Hodge
cycles \cite{B-P}, \cite{B-P-S}. Since we wish to present an
expository article we did choose to accompany the evolution of the
ideas from the geometric properties of algebraic families with
their singularities, to their representation by VMHS degenerating
at the discriminant locus of the family when the fibres acquire
singularities. In the last section we present a summary of the
results mentioned above.

 The study of morphisms in algebraic geometry is at the origin
 of the theory
of VMHS. To begin with smooth proper morphisms of smooth varieties
$f: X \to V$, the underlying differentiable structure of the
various fibers does not vary, by Ehreshman's theorem; the fibers
near a point of the parameter space $V$ are diffeomorphic in this
case, but the algebraic or analytic structure on the fibers do
vary.

 From another point of
view, locally, near a point on the parameter space, we may think
of a morphism as being given by a fixed differentiable manifold
and a family of analytic structures parameterized by the
neighborhood of the point.

 So the cohomology of the fibers does not vary, but, in general the
Hodge structure, which is sensitive to the analytic structure,
does. In this case, the cohomology groups of the fibers form a
local system $\LL_{\Z}$. So we start by the study of the structure
of local systems and its relation to flat connections
corresponding to the study of linear differential equations on
manifolds.  In the geometric case, the local system of cohomology
of the fibres define the Gauss-Manin connection.

The theory of variation of Hodge Structure (VHS) adds to the local
system the Hodge structures on the cohomology of the fibers, and
 transforms geometric problems concerning smooth proper morphisms
 into linear algebra problems involving the Hodge filtration by complex
 subspaces of cohomology vector spaces  of
the fibers. The data of VHS denoted by $(\LL, F)$ has three levels
of definition: the local system  of groups $\LL_{\Z}$, the Hodge
filtration $ F$ varying holomorphically with the fibers  defined
as a filtration by sub-bundles of $\LL_V := \OO_V \otimes
\LL_{\Z}$ on the base $V$, while the Hodge decomposition is on the
differentiable bundle $\LL_{\infty} := \CC_V^{\infty} \otimes
\LL_{\Z} = \oplus_{p+q = i}\LL^{p,q}$.

In general, a morphism onto the smooth variety $V$ is smooth
outside a divisor $D$ called its discriminant, in which case the
above description apply on the complement $V-D$, then the VHS is
said to degenerate along $D$ which means it acquires
singularities. The study  of the singularities may be carried in
two ways, either by introducing the theory of mixed Hodge
structure (MHS)on the singular fibers, or  by the study of the
asymptotic behavior of the $VHS$ in the neighborhood of $D$, but
in this case we may blow-up closed subvarieties in $D$ without
modifying the family on $V-D$, hence we may suppose $D$ a normal
crossing divisor (NCD) by Hironaka's results. We may also suppose
the parameter space  reduced to a disc and $D$ to a point, since
many arguments are carried over an embedded disc in $V$ with
center a point in $D$. In this setting, Grothendieck proved first
in positive characteristic that the local monodromy around points
in $D$ of a local system of geometric origin is quasi-unipotent.
Deligne explains a set of arguments to deduce geometric results on
varieties over a field in characteristic zero  from the case of
positive characteristic \cite{BBD}. Direct proofs exist using
desingularization and spectral sequences \cite {C}, \cite {L}.

 Technically it is easier to write this expository article if we
suppose the local monodromy unipotent, although this does not
change basically the results. \\
For a local system with unipotent local monodromy, we need to
introduce Deligne's canonical extension of the analytic flat
vector bundle $\LL_{V-D} $ into a bundle ${\LL}_V$ on $V$
characterized by the fact that the extended connection has
logarithmic singularities with nilpotent residues. It is on $\LL_V
$ that the Hodge filtrations $F$ will extend as a filtration by
subbundles, but they do not define anymore a Hodge filtration on
the fibres of the bundle over points in $D$. Instead, combined
with the local monodromy around the components of $D$ through a
point of $D$, a new structure called the limit mixed Hodge
structure (MHS)and the companion results on the Nilpotent orbit
and $SL(2)-$orbit \cite {Sc} describe in the best way the
asymptotic behavior of the VHS near a point of $D$.

The above summary is the background needed  to understand the
motivations, the definitions and the problems raised in the
theory.  That is why we recall in the first section, the relations
between local systems and linear differential equations as well
Thom-Whitney's results on the topological properties of morphisms
of algebraic varieties. The section ends with the definition of a
VMHS on a smooth variety.

  After introducing the theory of
mixed Hodge structure (MHS) on the cohomology of algebraic
varieties, Deligne proposed  to study the variation of such linear
MHS structure (VMHS) reflecting the variation of the geometry on
cohomology of families of algebraic varieties  ( \cite{WII}, Pb.
1.8.15).
 In the second section we study the properties of
degenerating geometric VMHS.

In the last section we give the definition and properties of
admissible VMHS and describe important local results of Kashiwara
\cite {K}. In this setting we recall the definition of normal
functions and we explain recent results on the algebraicity of the
zero set of normal functions to answer a question raised by
Griffiths and Green.

\vskip 0.2cm \centerline {Contents}

\n 1  Variation of mixed Hodge structures \hfill p 3\\
1.1 Local systems and representations of the fundamental
  group \hfill p 3\\
 1.2  Connections and Local Systems \hfill  p 4 \\
 1.3 VMHS of geometric origin \hfill p 7 \\
 1.4  Singularities of local systems \hfill p 10 \\
 2 Degeneration of VMHS \hfill  p 15\\
 2.1 Diagonal degeneration of geometric  VMHS \hfill p 15\\
2.2  Filtered mixed Hodge complex(FMHC) \hfill p 17\\
2.3 Diagonal  direct image of a simplicial cohomological FMHC
\hfill p 19 \\
2.4 Construction of a Limit MHS on the unipotent  nearby cocycles \hfill p 20\\
2.5 Smooth morphism \hfill p 21\\
2.6 Polarized Hodge-Lefschetz  structures  \hfill p 24\\
2.7 Quasi-projective case \hfill p 26\\
2.8 Alternative construction, existence and uniqueness \hfill p 27\\
3 Admissible variation of mixed Hodge structure \hfill p 29\\
3.1 Definition and results\hfill p 29\\
3.2 Local study of Infinitesimal Mixed Hodge structures \hfill p 30 \\
3.3 Deligne-Hodge theory on the cohomology of a smooth variety
\hfill p 32 \\
\section{ Variation of mixed Hodge structures}
The classical theory of linear differential equations on an open
 subset of $\C$ has developed into the theory of  connections on
manifolds, while the monodromy of the solutions developed into
representation theory of the fundamental group of a space.  With
the development of sheaf theory, a third definition of local
system as locally constant sheaves, appeared to be a powerful tool
to study the cohomology of families of algebraic varieties. In his
modern lecture notes \cite{HI} with a defiant classical title, on
linear differential equations with regular singular points,
Deligne proved the equivalence between these three notions and
studied their singularities. The applications in the study of
singularities of morphisms lead to the problems on degeneration of
VMHS.

\subsection{Local systems and representations of the fundamental
  group} We refer to  \cite{HI} for this section; the notion of local system
coincides with the theory of representations of the fundamental
group of a topological space.

In this section, we suppose the topological
space {\it $M$ locally path connected and locally simply connected}
(each point has a basis of connected neighborhoods $(U_i)_{i \in I}$
with trivial fundamental groups i.e $\pi_0 (U_i) = e $ and $\pi_1
(U_i) = e $).
In particular, on complex algebraic varieties, we refer
to the transcendental topology and not the Zariski topology to define
local systems.

\begin{defn}
  Let $\Lambda$ be a ring.  A local system of $\Lambda$-modules on a
  topological space $X$ is a sheaf $\mathcal{L}$ of $\Lambda_X$-modules
  on $X$ such that, for each $x\in X$, there is a neighborhood $U$ and a
  non-negative integer $n$ such
  that $\mathcal{L}_{|U}\cong \Lambda_U^n$.   A local system of
  $\Lambda$-modules is said to be \emph{constant} if it is
  isomorphic on $X$
  to $\Lambda_X^r$ for some fixed $r$.
\end{defn}

\begin{defn} Let $L$ be a finitely generated $\Z-$module. A
representation of  a group $G$  is a homomorphism of groups

\smallskip
\n  \centerline {$G \stackrel {\rho}{ \rightarrow} Aut_{\Z} (L )$}

\smallskip
\n from $G$ to the  group of $\Z-$linear automorphisms of $L$,
 or equivalently a linear action of $G$ on $L$.
 \end{defn}

We will also use the definition for $\Q-$vector spaces instead
of $\Z$-modules.

\subsubsection{Monodromy}
If $\mathcal{L}$ is a local system of $\Lambda$-modules on a
topological space $M$ and $f:N\to M$ is a continuous map, then
$f^{-1}(\mathcal{L})$ is a local system of $\Lambda$-modules on
$N$.
\begin{lem} A local system  $\mathcal L$  on the
interval $[0,1]$ is necessarily constant.
\end{lem}
\begin{proof}
Let $\mathcal{L}^{\hbox{\'et}}$ denote the \'etal\'e space of
$\mathcal{L}$. Since $\mathcal{L}$ is locally constant,
$\mathcal{L}^{\hbox{\'et}}\to [0,1]$ is a covering space.   Since
$[0,1]$ is contractible, $\mathcal{L}^{\hbox{\'et}}$ is a product.
This implies that $\mathcal{L}$ is constant.
\end{proof}
Let $\gamma: [0,1] \to M $ be a loop in $M$ with origin  a point
$v$ and let $\mathcal L$ be a $\Z-$local system on $M$ with fiber
$L$  at $v$. The inverse image $\gamma^{-1}({\mathcal L})$ of the
local system is isomorphic to the constant  sheaf defined by $L$
on $[0,1]$: $\gamma^{-1}{\mathcal L} \simeq L_{[0,1]} $, hence we
deduce from this property the notion of monodromy.
\begin{defn}[Monodromy]
  The composition of the linear
isomorphisms \\
\n  \centerline {$ L = {\mathcal L}_v = {\mathcal L}_{\gamma (0)}
\simeq \Gamma ([0,1],{\mathcal L} ) \simeq {\mathcal L}_{\gamma
(1)} = {\mathcal L}_v = L$}

\n is denoted by $T$ and  called  the monodromy along $\gamma$.
 It depends only on the homotopy class of $\gamma$.
 \end{defn}
\n The monodromy  of a local system ${\mathcal L}$ defines a
representation of the fundamental group  $\pi_1(M,v)$ of a
topological space $M$ on the stalk at $v$, ${\mathcal L}_v = L$

\smallskip \n \centerline
{$\pi_1(M,v)\stackrel {\rho}{ \rightarrow} Aut_{\Z} ({\mathcal
L}_v )$}

\smallskip \n which  characterizes  local systems on connected spaces in the
following sense
\begin{prop} Let $M$  be  a connected topological space. The
above correspondence  is an equivalence
 between the following categories\\
i)  $\Z-$local systems with finitely generated $\Z-$modules $L$ on
$M$\\
ii)  Representations of the fundamental group $\pi_1(M,v) $ by
linear automorphisms of finitely generated $\Z-$modules $L$.
  \end{prop}
\subsection{Connections and Local Systems}
 The concept of connections  on
analytic manifolds (resp. smooth complex algebraic variety) is a
generalization of the concept of system of $n-$linear first order
differential equations.
\begin{defn} Let $\FF$ be a locally free holomorphic
 ${\mathcal O}_X -$module
on a complex analytic manifold $X$ (resp. smooth algebraic complex
variety). A connection on $\FF$ is a $\C_X-$linear map

\smallskip
\n  \centerline { $\nabla: \FF \to \Omega^1_X \otimes_{{\mathcal
O}_X} \FF $ }

\n satisfying the following condition for all sections $f $ of
$\FF$ and $\varphi $ of $ {\mathcal O}_X $\ :

\smallskip
\n  \centerline { $ \nabla (\varphi f) = d \varphi \otimes f +
\varphi \nabla f$}

\n known as Leibnitz condition.
 \end{defn}
 We define a morphism of connections as
 a morphism of ${\mathcal O}_X-$modules which commutes with $\nabla$.
 \subsubsection{} The definition of $\nabla$ extends to
differential forms in degree $p$ as a $\C-$linear map

\smallskip \n \centerline { $ \nabla^p :\Omega^p_X \otimes_{{\mathcal O}_X} \FF
 \to
\Omega^{p+1}_X \otimes_{{\mathcal O}_X} \FF$ s.t. $\nabla^p
(\omega \otimes f) = d \omega \otimes f + (-1)^p \omega \wedge
\nabla f $}

\smallskip \noindent  {\em The connection is said to be integrable if its curvature
$ \nabla^{1} \circ
 \nabla^0: F \to \Omega^2_X \otimes_{{\mathcal O}_X} \FF$ vanishes}
 ($\nabla = \nabla^0$, and the curvature is a linear morphism).\\
  Then it follows that
the composition of maps $ \nabla^{i+1} \circ
 \nabla^i = 0$ vanishes for all  $i \in \N$ for an integrable
 connection.
 In this case a
  de Rham complex is associated to $ \nabla$

\smallskip \n \centerline { $ (\Omega^*_X \otimes_{{\mathcal O}_X} \FF, \nabla)
\colon= \FF \to \Omega^1_X \otimes_{{\mathcal O}_X} \FF \cdots
\Omega^p_X \otimes_{{\mathcal O}_X} \FF
\stackrel{\nabla^p}{\to}\cdots \Omega^n_X \otimes_{{\mathcal O}_X}
\FF$}
\begin{prop} The horizontal sections $\FF^\nabla$ of a connection $\nabla$ on
 a module $\FF$ on an  analytic (resp. algebraic) smooth variety $X$ ,
  are defined as the solutions of the differential equation on $X$
  (resp. on the analytic associated manifold $X^h$)

\smallskip
\n  \centerline { $\FF^\nabla = \{ f : \nabla (f) = 0 \}$.}

\n When the connection is integrable, $\FF^\nabla$ is a local
system of dimension $dim \, \FF$.
 \end{prop}
\begin{proof} This result is based on the relation between
differential equations and connections.
 Locally, we consider  a small open subset $U \subset X $  isomorphic
 to an open
set of $\C^n$ s.t. $ \FF_{\vert U}$ is isomorphic to ${\mathcal
O}_U^m$. This isomorphism is defined by the choice of a basis of
sections $( e_i )_{ i \in [1,m]}$ of $\FF$ on $U$ and extends to
the tensor product of $\FF$ with the module of differential forms:
$\Omega^1_U \otimes \FF \simeq (\Omega^1_U)^m$.\\
 In terms of the basis $e = (e_1,\cdots,e_m)$  of $F_{\vert U}$,
 a section $ s$ is written as

 \smallskip \n $ s = \sum_{i\in [1,m]}y_i e_i$ and
$\nabla s = \sum_{i\in [1,m]} dy_i \otimes e_i + \sum_{i\in
[1,m]}y_i \nabla e_i$ where \\
$\nabla e_i = \sum_{j\in [1,m]}\omega_{ij}\otimes e_j$.

\smallskip \n The connection
 matrix $\Omega_U$ is the matrix of differential forms
 $ ( \omega_{ij} )_{ i,j \in [1,m]}
 $, sections of $
\Omega^1_U$; its  $i-$th column is the transpose of the line image
of  $\nabla (e_i)$ in $(\Omega^1_U)^m$. Then the restriction of
$\nabla$ to $U$ corresponds to a connection on ${\mathcal O}_U^m$
denoted $\nabla_U$ and defined on  sections $y = (y_1,\cdots,y_m)$
of ${\mathcal O}_U^m $ on $U$, written in column as $\nabla_U
{^ty} = d ({^ty}) + \Omega_U \, {^ty}$ or

\smallskip \noindent \centerline {$\nabla_U$ \, $ \left(
\begin{array}{c}
y_1  \\
\vdots  \\
y_m  \\
\end{array}\right)$=$\left(
\begin{array}{c}
dy_1  \\
\vdots  \\
dy_m  \\
\end{array}\right)$ + $\Omega_U$ \, $\left(
\begin{array}{c}
y_1  \\
\vdots  \\
y_m  \\
\end{array}\right)$ }

\smallskip
 \n the equation is in $End (T, \FF)_{\vert U} \simeq (\Omega^1 \otimes \FF)_{\vert
 U}$ where $T$ is the tangent bundle to $X$.\\
  Let $(x_1,\cdots,x_n)$
 denotes the coordinates of $\C^n$,
 then  $\omega_{ij}$ decompose as

 \smallskip \noindent \centerline {$\omega_{ij} =  \sum_{k\in
 [1,n]}\Gamma^k_{ij}(x) \, dx_k$}

\smallskip \noindent so that the equation of the coordinates of horizontal
 sections is given by
linear partial differential equations for $i \in [1,m]$ and $k\in
[1,n]$
 $$\frac {\partial y_i}{\partial x_k} +
\sum_{j\in [1,m]}\Gamma^k_{ij} (x) \, y_j = 0$$
 \n The solutions form a local system of dimension $m$,
since the Frobenius condition is satisfied by the integrability
hypothesis on $\nabla$.
\end{proof}
 The connection appears as a
global version of linear differential equations, independent of
the choice of local coordinates on $X$.
\begin{rem} The  natural morphism  for a complex local system $ {\mathcal L} \to (\Omega^*_X
\otimes_{\C}{\mathcal L}, \nabla)$ defines a  resolution of
 ${\mathcal L}$ by coherent modules, hence  induces  isomorphisms on cohomology
$$ H^i (X,{\mathcal L}) \simeq H^i (R\Gamma(X, (\Omega^*_X
\otimes_{\C}{\mathcal L}, \nabla)))$$ where we take
hypercohomology on the right. On a smooth differentiable manifold
$X$, the  natural morphism $ {\mathcal L}\to (\EE^*_X
\otimes_{\C}{\mathcal L}, \nabla)$ defines a soft resolution of
${\mathcal L}$ and induces  isomorphisms on cohomology
$$ H^i (X,{\mathcal L}) \simeq H^i (\Gamma(X, (\EE^*_X
\otimes_{\C}{\mathcal L}, \nabla))$$
 \end{rem}

 \begin{thm}[Deligne] \cite{HI} The functor $(\FF,\nabla) \mapsto \FF^\nabla$ is
 an equivalence
 between the category of integrable connections on an analytic manifold $X$ and the category of
 complex local systems on $X$ with quasi-inverse defined by
 $ {\mathcal L} \mapsto {\mathcal L}_X $.
 \end{thm}

\subsubsection{ Local system of geometric origin} The structure of local
system appears naturally on the cohomology of a smooth and proper
family of varieties.
 \begin{thm}[differentiable fibrations ]  Let $f : M \to
N $ be a proper differentiable submersive morphism of manifolds.
For each point $v \in N$ there exists an open neighbourhood $U_v$
of $v$ such that  the differentiable structure of the inverse
image $M_{U_v} = f^{-1}(U_v)$
  decomposes as a  product of a fibre at $v$ with $U_v$:
$$f^{-1}(U_v) \stackrel{ \varphi }{ \stackrel {\simeq }{\longrightarrow}}
 U_v \times M_v \quad  \hbox {s.t. } \quad  pr_1 \circ \varphi = f_{\vert U_v}$$
\end{thm}
 The proof follows from the existence of a tubular
neighbourhood of the submanifold $M_v$ (\cite{Vo} thm. 9.3).
\begin{cor}
[Locally constant cohomology] In each degree $i$, the cohomology
sheaf of the fibers  $ R^i f_* \Z$ is constant on a small
neighbourhood $U_v$ of any point $v$ of fiber $H^i ( M_v, \Z)$ i.e
there exists an isomorphism between the restriction $(R^i f_*
\Z)_{\vert U_v} $ with the constant sheaf $ H^i_{U_v}$ defined on
$U_v$ by the vector space $H^i = H^{i}(M_v, \Z)$.
\end{cor}
\begin{proof} Let $U_v$ be isomorphic to a ball
 in $\R^n$ over which $f$
is trivial, then for any small ball $B_\rho$ included in $U_v$,
the restriction $H^{i}(M_{U_v}, \Z) \to H^{i}(M_{B_\rho}, \Z)$ is
an isomorphism since $M_{B_\rho}$ is a deformation retract of
$M_{U_v}$.
 \end{proof}

\begin{rem}[ Algebraic family of complex varieties] Let $f: X \rightarrow V$
be a smooth proper morphism  of complex algebraic varieties , then
$f$ defines a differentiable locally trivial fiber bundle on $V$.
 We still  denotes
by $f$  the differentiable morphism $ X^{dif} \to V^{dif} $
associated to $f$, then the complex of real differential forms
${\mathcal E}_X^*$ is a fine resolution of the constant sheaf $\R$
and $R^i f_* \R \simeq {\mathcal H}^i (f_* {\mathcal E}_X^*)$ is a
local system called of geometric origin. Such  local systems carry
additional structures and have specific properties which are the
subject of study in this article.
\end{rem}

 We give now the abstract definition of VMHS,
   (\cite{WII},1.8.14) and then explain how the geometric
   situation leads to such structure.

\begin{defn} A VMHS on an analytic manifold $X$ consists of\\
1) A local system $\LL_{\Z}$ of $\Z-$modules of finite type,\\
 2) A finite increasing filtration $\WW $ of $\LL_{\Q}:= \LL_{\Z} \otimes \Q$
 by sublocal systems of rational vector spaces,\\
  3) A finite decreasing filtration $\FF$  by locally free analytic
  subsheaves
   of $ \LL_X := \LL_{\Z} \otimes  \OO_X $ whose
sections on $X$ satisfy the infinitesimal (Griffiths)
transversality relation with respect to the  connection $\nabla$
defined by the structure of local system on $\LL_{\C} :=
\LL_{\Z}\otimes \C$ on
 $ \LL_X $
\[ \nabla (\FF^p ) \subset \Omega^1_X \otimes
_{\OO_X}\FF^{p-1}
\]
such that $\WW $ and $\FF$ define a $MHS$ on each fiber
 ($ \LL_X (t), \WW (t), \FF(t)$) at $t$ of
the bundle ~$ \LL_X$.
\end{defn}
 The definition of VHS is obtained in the particular pure  case when the
 weight filtration is constant but for one index. The induced filtration by
$\FF$ on the graded objects $Gr^{\WW}_m \LL$ form a VHS.  A
morphism of VMHS is a morphism of local systems compatible with
the filtrations.

\begin{defn} The VMHS is graded polarizable if
the graded objects
  $Gr^{\WW}_m \LL$ are
polarizable variation of Hodge structure.
\end{defn}

 \subsection{VMHS of geometric origin}
In the above situation of smooth algebraic morphisms, the
cohomology of the fibers carry  a Hodge structure (HS) which leads
to the theory of variation of Hodge structure (VHS) on an
underlying local system and which is the subject of another
course.  We describe here structural theorems of algebraic
morphisms and as a consequence the variation of mixed Hodge
structure (VMHS) they define on the cohomology of the fibers over
strata of the parameter space. The asymptotic  properties of a
VMHS near the boundary of a strata is studied under the
terminology of degenerating VMHS and will be discussed here.

  The
study of the whole data, including many strata has developed in
the last twenty years after the introduction of perverse sheaves.
In the projective case, a remarkable decomposition result is
proved in \cite{BBD}. In the transcendental case, this result
apply for Hodge differential modules \cite{Sa}. These results are
beyond the scope of this article. Instead we discuss preliminary
results needed to understand such theory.

 \subsubsection{Background on morphisms of algebraic varieties and
 local systems} We describe here structural theorems of algebraic
morphisms in order to deduce later VMHS on various strata of the
parameter space.

Stratification theory on a variety consists of the decomposition
of a  variety into the disjoint union of smooth locally closed
algebraic (or analytic)  subvarieties  called strata ( a strata is
smooth but the variety may be singular along a strata). By
construction, the closure of a strata is a union of additional
strata of lower dimensions. A Whitney stratification satisfies two
more conditions named after H.~Whitney. We are interested here in
their consequence, after the work of J. Mather: the local
topological trivial property at any point of a strata that will be
useful in the study of local cohomology. Thom described in
addition the topology of the singularities of algebraic morphisms.
Next, we summarize these results.
 \subsubsection{Thom-Whitney's stratifications}
 Let $f: X \rightarrow V$ be an
algebraic morphism. There exist finite Whitney stratifications
 ${\mathfrak X}$ of $X$
and $\mathcal {S} = \{S_l\}_{l\leq d}$
 of  $V$ (dim. $S_l$ = $l$, dim. $V$ = $d$) such that for each connected
component S of an $\mathcal {S}$  stratum $S_l$ of $V$\\
  i)$f^{-1}S$ is a  topological fibre bundle over $S$, union of
connected components of strata of ${\mathfrak X}$, each mapped
submersively to~$S$.\\
 ii) Local topological triviality: for all
$v \in S$, there exist an open neighborhood $U(v)$ in $S$ and a
stratum preserving homeomorphism $h: f^{-1} (U) \simeq f^{-1} (v)
\times U $ s.t. $ f|U  = p_U \circ h $ where $p_U$ is the
projection
on $U$.\\
 This statement can be
found in an article by D. T. L\^e  and B. Teissier  \cite{L-T} and
\cite{g-m} by Goretsky and MacPherson.
 Since the restriction  $f/S $ to a stratum $S$ is a locally
trivial topological bundle, we  deduce
\begin{cor} For each integer $i$, the  higher direct cohomology
sheaf $(R^i f_* \Z_X)/S$  is locally constant on each stratum $S$
of $V$.
\end{cor}
 Then we say that $R^i f_* \Z_X$ is constructible
on $V$ and   $ R f_* \Z_X$ is cohomologically constructible on
$V$.

 \subsubsection{Geometric Variation of
Mixed Hodge Structures} The abstract definition of VMHS above
summarizes in fact properties of  the variation of MHS  defined on
the cohomology of the fibers over strata of the parameter space.
We suppose next the parameter space a   complex disc $D$ that we
can suppose small enough to have a topological fibration on the
punctured disc $D^*$. Hence the cohomology  groups $( H^i(
X_t,\Q))$ form a local system $ (R^i f_* \Q_X)$ on $D^*$ with an
associated flat  analytic bundle $(R^i f_* \C_X)/D^* \otimes
\OO_{D^*} $ endowed with an analytic connection $\nabla$ whose
flat sections form a local system isomorphic to $R^i f_*
\C_X)/D^*$.

 Suppose the fibers of $f$ are algebraic
varieties, then a $MHS$ exists on the cohomology  groups $( H^i(
X_t,\Z), W, F)$ of the fibers $X_t$.
 The following proposition
describes properties of the weight and Hodge filtrations: the
variation of the weight filtration $W$ is locally constant in $t$
and the
 variation  of the  Hodge filtration $F$ is  analytic in $t$.
\begin{prop}
 Suppose the fibers of the above morphism $f: X \to D$
are algebraic and the radius of $D$  small enough \\
i) For all integers $i \in \N$, the restriction to $D^*$ of the
higher  direct image cohomology $ (R^i f_* \Z_X)/D^*$ (resp. $ R^i
f_* \Z_X/ {\rm Torsion})/D^*$ are local systems of
free $\Z_{D^*}-$modules of finite type.\\
ii) The weight filtration $W$ on the cohomology $H^i(X_t, \Q)$ of
a fiber $X_t$ at $t \in D^*$ define a filtration $\WW$ by sublocal
systems of $
(R^i f_* \Q_X)/D^*$.\\
 iii) The Hodge filtration $F$ on the cohomology $H^i(X_t, \C)$
  define a filtration $\FF$ by analytic sub-bundles  of $
(R^i f_* \C_X)/D^* \otimes \OO_{D^*} $ whose  locally free sheaf
of sections on $D^*$ satisfy the infinitesimal Griffiths
transversality with respect to the (Gauss Manin)connection
$\nabla$
\[ \nabla \FF^p \subset \Omega^1_{D^*} \otimes
_{\OO_{D^*}}\FF^{p-1}.
\]
iv) If we suppose $f$ projective, then the induced filtration by
$\FF$ on the graded objects $Gr^{\WW}_m (R^if_* \C_X)/D^*$ are
polarizable variation of Hodge structure (VHS).
\end{prop}
The proposition is a generalization to the non proper case of
results in the smooth proper case. The proof follows the
historical developments of the theory, and it is in two steps. In
the first step we summarize the results for $VHS$ and in the
second step we use the technique introduced by Deligne by covering
$X$ by simplicial smooth varieties \cite {HIII}.

\subsubsection*{$VHS$ defined by a smooth proper morphism}
The main point is to prove  that the variation of the Hodge
filtration is  analytic. The original proof by Griffiths  is based
on the description of the Hodge filtration $F$ as a map to the
classifying space of all filtrations of the cohomology vector
space
 of a fiber at a reference point. Here the Hodge  filtration of
the cohomology at a fiber $y$ are transported horizontally to the
reference point.

 We summarize here a proof based on the use of
relative connections by Deligne (\cite{HI}, 2.20.3), and Katz-Oda
\cite{K-O}. Let $f: X \to V$ denotes a smooth morphism of analytic
varieties, $f^{-1}\FF$ the sheaf theoretic inverse of a sheaf
$\FF$ on $V$, and let $\Omega^1_{X/V}$ denotes the sheaf of
relative differential forms.

\begin{defn} i) A relative connection on  a coherent sheaf
   of modules on $X$
 \[\nabla: \VV \to \VV \otimes \Omega^1_{X/V} \]
 is defined by an $f^{-1}\OO_V$ linear map $\nabla$ satisfying for
 all local sections
\[ f \in \OO_X, v \in \VV: \quad \nabla (f v)=
 f.\nabla v + df. v \]

ii) A relative local system on $X$ is a sheaf $\LL$ with a
structure of $f^{-1}\OO_V-$module, locally isomorphic on $X$ to
the inverse image of a coherent sheaf on a subset of $V$
\[ \forall x \in X, \,\exists x \in U \subset X, W \subset S: f_{|U}:U \to W,
\, \exists \,  \FF_W  \, \hbox{a coherent sheaf}:
 \LL_{|U} \simeq (f^{-1} \FF)_{|U}.
 \]
\end{defn}
Then the theory is similar to the absolute case and there is an
equivalence between relative local systems and flat relative
connections. Moreover, the relative de Rham complex
$\Omega^*_{X/V}$ is defined in the flat case and the flat sections
of $\VV$ form a relative local system.

\begin{ex} Let $\LL_{\Z}$ be a local system on $X$, then
$\LL_{\hbox{rel}} := f^{-1}\OO_V \otimes \LL_{\Z}$ is a relative
local system and $\Omega^*_{X/V} \otimes \LL_{\Z}$ is its de Rham
resolution. In particular, we have  (\cite {HI}, 2.27.2):
 \[ \OO_V \otimes R^if_* \LL_{\Z} \simeq R^i f_* (\Omega^*_{X/V}
  \otimes \LL_{\Z}) \]
\end{ex}

To prove that the Hodge filtration vary holomorphically with the
paramaters and that the connection satisfy the infinitesimal
transversality, we consider the exact sequence of differential
forms
 \[
   0 \to f^* \Omega_V^1 \to \Omega_X^1 \to \Omega_{X/Y}^1 \to 0,
\]
 and the exact
sequence of complexes
 \[  0 \to f^* \Omega_V^1\otimes  F^{p-1}\Omega_X^{*-1}
 \to F^p \Omega_X^*  \to F^p \Omega_{X/Y}^* \to 0
 \]
taking the higher direct image, Katz-Oda \cite {K-O} prove that
the associated connecting morphism
 \[ R^i f_* F^p \Omega_{X/Y}^* \to \Omega_V^1 \otimes F^{p-1} R^i f_*
 \Omega_{X/Y}^*\]
 coincide with the connection. More generally, there exists a
  filtration $ L^r \Omega^i_X:= f^* \Omega_V^r
\otimes \Omega^{i-r}_X \subset \Omega^i_X  $ of the complex
$\Omega^*_X  $.

\subsubsection*{General case} The proof of the proposition is  deduced from the
smooth proper case via Deligne's simplicial resolutions \cite
{HIII}: there exists  a smooth simplicial variety $ X_*$ defined
by a family $\{X_n\}_{n \in \N}$ over $X$ defining a cohomological
hyperresolution of $X$, which gives in particular an isomorphism
between the cohomology groups of the simplicial variety $X_*$ and
$X$. Applying the structural theorem to the morphisms $f_i: X_i
\to D$, we deduce that the restriction of $f_i$ to $D^*$ is smooth
so that we can use the relative de Rham complex of each $X_i$ over
$D^*$ as in the above case. \\
We do not give the details of the proof here, since we will treat
a similar case later to study the asymptotic behavior of the VMHS
near the origin of the disc, in presence of singularities of $f$
on the fiber at the origin. Precisely, there are two different
cases. In the first case of a proper  morphism $f: X \to D$, the
varieties are proper over $D$.  In the second case when $f$ is not
proper, but with   algebraic fibers, the construction of MHS is
based on the completion of $X$ by a divisor $Z$ over which the
varieties $X_i$ of $X_*$ are completed by  normal crossing
divisors
 $Z_i$, so we can use  the de Rham complex with logarithmic
singularities along $Z_i$.
 The key point here is that the NCD form a family of relative NCD
 over $D^*$ for $D$ small enough,
for a finite number of indices which determine the cohomology.

\begin{cor} For each integer $i$ and each stratum $S \subset V$
 of a Thom-Whitney stratification of $f: X \to V$, the restriction of the
  higher direct cohomology
sheaf $(R^i f_* \Z_X)/S$  underly a VMHS on $S$.
\end{cor}
{\it Proof}. By definition of $S$, the restriction of the higher
direct image
 is a local system $\LL_i$.  By the above argument via
simplicial coverings, for any embedded
 disc in $S$ with center a point $v \in S$, we know only that
the restriction to $D^*$ is a VMHS. We prove that this VMHS extends
across $v$. \\
 Since the local system
extends to the whole disc, the local monodromy is trivial at $v$,
but we don't know yet about the extension of the weight and the
Hodge filtrations. Since the local monodromy on the weight local
subsystems $\WW_r$ is induced from $\LL_i$, it is trivial at $v$
and $\WW_r$ extend at $v$.

The proof of the extension of the Hodge filtration to an analytic
bundle over $D$ can be deduced from the study of the asymptotic
behavior of the Hodge filtration in the next section. The
extension of the Hodge filtration has been proved for the proper
smooth case \cite{S}, \cite{Sc}.

 In the next section we will define the limit MHS of the VMHS
on $D^*$, and  a comparison theorem via the natural morphism from
the MHS on the fiber $X_v$ to this limit MHS, is stated  as the
basic local invariant cycle theorem (\cite{G},VI). Since the local
monodromy is trivial in our case, both MHS coincide above the
point $v$, hence the Hodge filtration extends also by analytic
bundles over $D$ since it is isomorphic to the limit Hodge
filtration.

\subsection{Singularities of local systems } In general the local system
$(R^i f_* \Z_X)/S $ over a strata $S$ does not extend as a local
system to the boundary of $S$.

The study of the singularities at the boundary may be carried
through the study of the singularities of the associated
connection. It is important to distinguish in the geometric case
between the data over the closure $\partial S :=\overline {S} - S$
of $S$ and the data that can be extracted from the asymptotic
behavior on $S$, which are linked by the local invariant cycle
theorem.

The degeneration can be studied locally at points of $\partial S$
  or globally, in which case  we suppose  the boundary
 of $S $ a  NCD, since we are often reduced to such case by the
 desingularization theorem of Hironaka.

We discuss in this section, the quasi-unipotent property of the
monodromy and
 Deligne's  canonical extension of the connection.

\subsubsection{Local monodromy}
To study the local properties of $ R^if_* \Z_X$ at a point $v\in
V$, we consider an embedding of a small disc $D$ in $V$ with
center  $v$. Then, we reduce the study to the case of a proper
analytic morphism $f: X \to D$ defined on an analytic space to a
complex disc $D$. The inverse image $f^{-1}(D^*)$ of the punctured
disc $D^*$, for $D$ small enough, is a topological fiber bundle
(\cite{SGA7}, Exp. 14, (1.3.5)). It follows that a monodromy
homeomorphism is defined on the fiber $X_t$ at a point $t \in D^*$
by restricting $X$ to a closed path $\gamma: [0,1] \to D^*, \gamma
(u) = exp(2i \pi u)t $. The inverse fiber bundle  is trivial over
the interval: $\gamma^* X \simeq [0,1] \times X_t $ and the
monodromy  on $X_t$ is defined by the path $\gamma $ as follows
\[ T: X_t \simeq (\gamma^* X)_0\simeq (\gamma^* X)_1 \simeq
X_t \] The monodromy is independent of the choice of the
trivilization, up to homotopy, and can be achieved  for singular
$X$ via the integration of a special type of vector fields
compatible with a Thom stratification of  $X$ \cite{ V}.
\begin{rem} The following construction, suggested in the introduction of
(\cite{SGA7}, Exp.13, Introduction) shows how $X$ can be recovered
as a topological space from the monodromy. There exists a
retraction $r_t: X_t \to X_0$ of the general fiber $X_t$ to the
special fiber $X_0$ at $0$, satisfying $r_t \circ T = r_t$, then
starting with the system
$(X_t, X_0, T, r_t)$ we define\\
i) $X'$ and $f:X' \to S^1$ by gluing the boundaries of $X_t \times
[0,1]$, $X_t \times 0$ and $X_t \times 1$ via $T$. A map $r': X'
\to X_0$ is deduced from $r_t$.\\
ii) then $f: X \to D $ is defined as the cone of $f'$
\[ X \simeq X' \times [0,1]/(X' \times 0 \xrightarrow{r'} X_0) \to
S^1 \times [0,1]/(S^1 \times 0 \to 0) \simeq D \]
\end{rem}
\subsubsection{Quasi-unipotent monodromy}
If we suppose  $D^* \subset S $, then the monodromy of the local
system $(R^i f_* \Z_X)/D^* $ along a generator of $D^*$, as a
linear operator on the cohomology $H^i(X_t, \Q)$ may have only
roots of unity as eigenvalues. This condition discovered first by
Grothendieck for algebraic varieties over a field of positive
characteristic, is also true over $\C$ \cite {G-S}.
 \begin{prop} The monodromy $T$ of the local
system $(R^i f_* \Z_X)/D^* $ defined by  an algebraic (resp.
proper analytic) morphism $f:X \to D$ is quasi-unipotent
\[ \exists \, a, b \in \N \quad {\rm s.t.}  \quad (T^a - Id)^b = 0 \]
\end{prop}
A proof in the analytic setting is explained in \cite{C} and \cite
{L}. It is valid for an abstract VHS (\cite {Sc}, 4.5) where the
definition of the local system over $\Z$ is used in the proof;
hence it is true for VMHS. Finally, we remark that the theorem is
true also for the local system defined by the Milnor fiber.
\subsubsection{Universal fiber} A canonical way to study the general fiber,
independently of the choice of $t \in D^*$, is to introduce the
universal covering $\widetilde D^*$  of $D^*$ which can be defined
by Poincar\'e half plane $\H = \{ u \in \C: \hbox{Im} \, u > 0 \},
\,\pi: \H \to D^* : u \mapsto exp \, 2i \pi u$. The inverse image
$\widetilde X^* := \H \times_D X $ is a topological fibre bundle
trivial over $\H$ with fibre homeomorphic to $X_t$. Let $H:
\H\times X_t \to \widetilde X^* $ denotes a trivialization of the
fiber bundle with fiber $X_t$ at $t$,  then the translation $ u
\to u+1$ extends to $\H \times X_t$ and induces via $H$, a
transformation $ T $ of $\widetilde X^*$ s.t. the following
diagram commutes
$$
\begin{array}{ccc}
X_t&\xrightarrow{I_0}&\widetilde X^*\\
\downarrow T_t&&\downarrow{ T }\\
X_t&\xrightarrow{I_1}&
\widetilde X^*\\
\end{array}
$$
where $I_0$ is defined by the choice of a point $u_0 \in \H$
satisfying  $e^{2i\pi u_0} = t$ s.t. $I_0: x \mapsto H(u_0, x)$,
then $I_1: x \mapsto H(u_0 + 1, x)$ and $T_t$ is the monodromy on
$X_t$. Hence $T$ acts on $\widetilde X^*$ as a universal monodromy
operator.

\subsubsection{Canonical extension with logarithmic singularities}
 For the global study of the asymptotic behavior of the local system on
 a strata $S$ near $\partial
S$, we introduce a general construction by Deligne \cite {HI} for
an abstract local system.

\subsubsection*{Hypothesis}   Let $Y$ be a normal crossing divisor (NCD)
 with smooth  irreducible components $Y =
\cup_{i\in I} Y_i$    in a smooth analytic variety $X$, $j: X^* :=
X-Y \to X$  the
  open embedding, and  $\Omega^*_X(log Y)$ the complex of sheaves
 of differential forms  with logarithmic poles  along $Y$. It is a
 complex of subsheaves  of $j_* \Omega^*_{X-Y} $.
 There exists a global residue
morphism $Res_i: \Omega^1_X(log Y)\to
 {\mathcal O}_{Y_i}$ defined locally as follows: given a point $y
 \in Y_i$ and $z_i$ an analytic local coordinate  equation of $Y_i$ at $y$,
 a differential  $\omega = \alpha \wedge d z_i /z_i + \omega'$
 where $\omega'$  does not contain $dz_i$ and $\alpha$ is regular
 along $Y_i$, then $Res_i (\omega) = \alpha_{|Y_i}$.

\begin{defn}Let  $F$ be a vector bundle on $X$. An analytic connection on $F$ has
logarithmic poles along $Y$ if the entries of the connection
matrix are one forms in $\Omega^1_X (Log Y)$, hence define a $\C-$
linear map satisfying Leibnitz condition
$$\nabla:
 F \to \Omega^1_X (Log Y) \otimes_{{\mathcal O}_X} F$$
\end{defn}
Then  the definition of $\nabla$  extends to
$$\nabla^i: \Omega^i_X
( Log Y) \otimes_{{\mathcal O}_X} F \to \Omega^{i+1}_X ( Log Y)
\otimes_{{\mathcal O}_X} F$$
 it is integrable if $\nabla^1 \circ
\nabla = 0 $, then a logarithmic complex $\Omega^*_X ( Log Y)( F)
\colon= (\Omega^*_X ( Log Y)
\otimes_{{\mathcal O}_X} F, \nabla)$ is defined in this case.\\
 The composition of $\nabla$ with the residue map:
  $$ (Res_i \otimes Id) \circ \nabla: F \to
 \Omega^1_X (Log Y) \otimes_{{\mathcal O}_X} F \to
  {\mathcal O}_{Y_i}\otimes F $$
vanishes on the product ${\mathcal I}_{Y_i} F$ of $F$ where
${\mathcal I}_{Y_i} $ is the ideal of definition of  $Y_i$. It
induces a linear map:
\begin{lem}(\cite {HI},3.8.3) The residue  of the connection is a linear
endomorphism of analytic bundles on $Y_i$
 $$ Res_i (\nabla): F \otimes_{{\mathcal O}_X} {\mathcal O}_{Y_i}
 \to F \otimes_{{\mathcal O}_X} {\mathcal O}_{Y_i}.$$
 \end{lem}

\begin{thm}[Logarithmic
extension] (\cite {HI}, 5.2) Let  ${\LL}$ be a complex local
system on the complement of the NCD: $Y$ in $X$ with locally
unipotent monodromy along the components $Y_i, i \in I$ of $Y$.
Then there exists a  locally free module $\LL_X$ on $X$ which
extends $\LL_{X^*}:= \LL \otimes \OO_{X^*}$, moreover the
extension is unique if the connection $\nabla$ has logarithmic
poles with respect to $\LL_X$
 \[\nabla \colon \LL_X
\rightarrow {\Omega}^1_X (Log Y) \otimes \LL_X \]
 with nilpotent residues along $Y_i$.
 \end{thm}
\begin{proof}
  The local system ${\mathcal L}$
 is  {\it locally unipotent along $Y$ if at any point $y \in Y$ all $T_j$
  are
 unipotent}, in which case the extension we describe is {\it canonical}.
 The construction has two steps, the first describes a local extension
 of the bundle, the second consists to show that local coordinates patching    \\
of the bundle over $X^*$ extends to a local coordinates patching
of the bundle over $X$. The  property of flat bundles is important
here since it is not known how to extend any analytic bundle on
$X^*$. We rely on a detailed exposition of Malgrange \cite{M}. We
describe explicitly the first step, since it
will be useful in applications. \\
 Let $y$ be a point in $Y$, and let $U(y)$ be a polydisc $D^{n} $ with
center $y$ and $\LL \to U(y)^*:= (D^*)^p \times D^{n-p} $ the
restriction of the local system.
 The universal covering $\widetilde U(y)^*$ of $(D^*)^p\times D^{n-p} $
 is defined by
 \[
 \H^{n-p}\times D^p = \{t = (t_1, \cdots, t_n) \in \C^n : \forall i \leq p, \, \hbox { Im}\, t_i > 0
 \, \hbox { and }\,
 \forall i > p, |t_i| < \varepsilon\}
 \]
 where the covering map is
\[ \pi: \H^{n-p}\times D^p \to (D^*)^{n-p}\times D^p:
 t \to (e^{2i \pi t_1},\cdots, e^{2i \pi t_p}, t_{p+1},
\cdots,t_n). \]

  We fix a reference point  $t_0
\in U(y)^*$ so that the local system is determined by a vector
space $L$ with the action of the various  monodromy  $T_j$ for $j
\leq p$ corresponding to the generators $\gamma_j$  around $Y_j$
of the fundamental group of $U(y)^*$.

The  inverse image $\widetilde \LL := \pi^*(\LL_{|U(y)})$ on
$\widetilde U(y)^*$
  is trivial with global sections a vector space  $\widetilde L $ isomorphic
  to $L$.
The action of the monodromy  $T_j$ for $j \leq p$  on $\widetilde
L $ is  defined by the formula:

 \[  T_j \, v (t) = v(t_1,
\cdots,t_j + 1, \cdots, t_n), \quad \hbox {for a section } v \in
\widetilde L.
\]
 We define  $N_j = - \frac{1}{2i \pi } Log T_j = \frac{1}{2i \pi } \sum _{ k
> 0}(I-T_j)^k/k$ as a nilpotent endomorphism of $\widetilde L$.
We consider the following embeddedding of the  vector space
$\widetilde L$ of multivalued sections of
   ${\LL} $ into the subspace of  analytic sections of the sheaf
   $ \LL_{\widetilde U(y)^*} $
   on $\widetilde U(y)^*$ by the formula:
\[ \widetilde L \to  \LL_{\widetilde U(y)^*}: v \mapsto \widetilde {v}
= (exp ( 2i \pi \Sigma_{i \leq p}   t_i N_i)). v
\]
 Notice that the  exponential is a linear sum of
  multiples of $Id - T_j$ with analytic coefficients, hence its
  action  defines an analytic section.
\begin{lem} Let $j_y: U(y)^* \to U(y)$, then the analytic section
$\widetilde v \in \LL_{\widetilde U(y)^*}$ descends to an analytic
section of $j_{y,*}\LL_{U(y)^*}$.
\end{lem}

 We show $\widetilde {v} (t+e_j) = \widetilde {v}
 (t)$ for all
$  t \in \widetilde {U}(y)^*$ and all vectors $e_j = (0, \cdots,
 1_j, \cdots, 0)$. We have\\
 $ \widetilde {v} (t+e_j) =
( [exp ( 2i \pi \Sigma_{i \leq p}  t_{i} N_{i}]exp (2i \pi
N_j).v)(t+e_j) $ \\
 where \\
 $(exp (2i \pi N_j).v)(t+e_j) = (T_j^{-1}.v)(t+e_j) = (T_j^{-1} T_j.v )(t)
 = v(t)$\\
 since $(T_j.v )(t) = v(t+e_j)$.

 {\it  The bundle $\LL_X$ is defined by the locally free subsheaf  of
 $j_* \LL_{X^*} $ with fibre $\LL_{X,y}$ generated as an $\OO_{X,y}-$module
 by the  the sections
   $\widetilde {v}$ for $v \in  \widetilde L$. }

   In terms of the local coordinates $z_i, i \in [1,n]$ of $U(y)$,
   the analytic function on $U(y)^*$  defined by a vector $v \in L$
   is given by the formula
\[
  \widetilde v (z)  = (exp ( \Sigma_{i \leq p} ( log z_{i}) N_{i})). v
 \]
  where $2i \pi t_i= log z_i $ is a determination of the logarithm, moreover $ \nabla
\widetilde {v} = \Sigma_{i \leq p} \widetilde { N_{i}.v } \otimes
\frac {dz_i}{z_i}. $
 \end{proof}
\begin{rem} i) The residue along a component $Y_i$ of the logarithmic
connection $\nabla$ is a nilpotent linear endomorphism of the
analytic bundle $ \LL_{Y_i} := \LL_X \otimes \OO_{Y_i}$
\[ \NN_i:= Res_i (\nabla): \LL_{Y_i}  \to \LL_{Y_i}, \quad  \]
ii) Let $Y_{i,j} := Y_i\cap Y_j$, then the restrictions of $
\NN_{Y_i}$ and $\NN_{Y_j}$ commute (\cite {HI} 3.10).\\
 iii) Let $Y_i^* := Y_i - \cup_{k \in I-i}Y_i \cap Y_k$.
 There is no global local system underlying  $\LL_{Y_i^*}$.
 Locally, at any point $y \in Y_i$,  a section
 $ s_i: Y_i \cap U(y)^* \to U(y)^*$ may be defined by the  hyperplane
 parallel to $Y_i \cap U(y)^* $ through the reference point $t_0$.
 Then $\LL_{Y_i}$ is isomorphic to the
  canonical extension of the inverse local system
  $s_i^* \LL$ (\cite {HI}, 3.9.b).

  Equivalently, the formula: $T_i= exp (-2 i\pi Res_i (\nabla))$
  (\cite {HI}, 1.17, 3.11)
holds on  the sheaf $\Psi_{z_i} \LL$ of nearby
  cycles.
\end{rem}

 \subsubsection{Relative monodromy filtration}
In the above formula the endomorphism $T^a - Id$ is nilpotent.
This was one of  the starting point
 to study the degeneration, when  Deligne introduced a monodromy
filtration satisfying some kind of degenerating Lefschetz formula
and representing the Jordan form of a nilpotent endomorphism
 by a canonical filtration.

 Let $V$ be a vector space, $W$ a finite increasing
filtration of $V$ by subspaces and $N$ a nilpotent endomorphism of
$V$ compatible with $W$, then Deligne proves in (\cite{WII}, prop.
1.6.13)
\begin{lem} There exists at most
a unique increasing filtration $M$ of $V$ satisfying
 $ N M_i \subset M_{i-2} $ s.t. for all integers $k$ and $l \geq
0$
 \[
N^{k}:Gr^M_{k+l}VGr^W_{l}V \xrightarrow{\sim} Gr^M_{l-k}Gr^W_l V
\]
\end{lem}
\begin{rem}  The lemma is true  for an object $V$ of
 an abelian category  $\A$, a finite increasing filtation $W$ of $V$
 by subobjects of
$V$ in $\A$ and a nilpotent endomorphism  $N$. Such generalization
is particularly interesting when applied to the abelian category
of perverse sheaves on a complex variety.
\end{rem}
\begin{defn}  When it exists, such  filtration is called the
relative monodromy filtration of  $N$ with respect to $W$ and
denoted by $M(N,W)$.
 \end{defn}
\begin{lem} If the filtration $W$ is trivial of weight $a$,
the relative monodromy filtration $M$  exists on the vector space
$V$ and satisfy
 \[ N M_i \subset M_{i-2}; \quad
N^{k}:Gr^M_{a+ k}V \xrightarrow{\sim} Gr^M_{a -k} V
\]
\end{lem}
The above definition has a striking similarity with Hard Lefschetz
theorem.
\begin{ex} Let $V = \oplus_i H^i(X, \Q)$ denotes the direct sum of
cohomology spaces of  a smooth  projective complex variety $X$ of
dimension  $ n$ and $N = \cup c_1 (H)$ the nilpotent endomorphism
defined by the cup product with the cohomology class of an
hyperplane section $H$ of $X$. Consider the increasing filtration
of $V$
$$ M_i(V) = \oplus_{j \geq n - i} H^j(X, \Q)$$
By hard Lefschetz theorem, the repeated action  $ N^{i}:Gr^M_{i}V
\simeq H^{n-i}(X, \Q) \to Gr^M_{-i} V \simeq H^{n+i}(X, \Q)$
 is an isomorphism. Hence $M_i$ coincides
 with the the monodromy filtration
 defined by $N$ and centered at $0$.
Following this example, the  property of the relative monodromy
filtration appears as a  degenerate  form of Lefschetz result on
cohomology.
\end{ex}
{\it The monodromy filtration centered at $0$}. Let $N$ be
nilpotent on $V$ s.t. $N^{l+1} = 0$. This is the case  $W$ is
trivial s.t. $W_0 = V$ and $W_{-1} = 0$, then $M$ is constructed as follows. Let $M_l := V$, $M_{l-1}:=\ke N^l$  and $M_{-l} := N^l(V)$, $M_{-l-1} := 0$ then $N^l: Gr^lV \simeq Gr^{-l}V $ and  the induced morphism $N'$ by $N$ on the quotient space $ V/ M_{-l}$  satisfy $(N')^l = 0$ so that the definition of $M$ is by induction on the index of nilpotency. The
primitive part of $ Gr^M_{i}V$ is defined for $i \geq 0$ as
\[P_i := \ke N^{i+1}: Gr^M_i \to Gr^M_{-i-2}, \quad Gr^M_i V \simeq \oplus_{k \geq 0} N^k
P_{i+2k} \] The decomposition at right
 follows from this definition, and the proof is similar to the existence of
 a primitive decomposition following the hard Lefschetz theorem on
 compact k\"{a}hler manifolds. In this case the  filtration $M$
  gives a description  of the Jordan form
of the nilpotent endomorphism $N$, independent of the choice of a
Jordan basis as follows. 
For all $i\geq 0$, we have the following properties:\\
- $N^i: M_i \to M_{-i}$ is surjective, \\
- $ \ke N^{i+1} \subset M_i $ and\\
- $ \ke N^{i+1}$ projects surjectively onto the primitive
subspace $P_i \subset Gr^M_i $.\\
 Let  $(e_i^j)_{j \in I_i} $ denotes a subset of elements in $ \ke N^{i+1}$
which lift a basis of $P_i \subset Gr^M_i V$, then the various
elements $ N^k (e_i^j)_{j \in I_i, 0\geq 0 k \leq i}$ define a
Jordan basis of $V$ for $N$. In particular, each element
$(e_i^j)_j$ for fixed $i$ gives rise to a Jordan block of length
$i+1$ in the matrix of $N$.

\subsubsection{Limit Hodge filtration} Let $(\LL, F)$ be an abstract
polarized VHS  on a punctured disc $D^*$, where the local system
is defined by a unipotent endomorphism $T$ on a $\Z-$module $L $;
then W. Schmid \cite {G-S} showed that such VHS is asymptotic to a
`` Nilpotent orbit'' defined by a
 filtration $F$ called limit or asymptotic such that
 for $N = Log T$  the nilpotent logarithm of $T$, and  $W(N)$ the monodromy
 filtration, the data $(L, W(N), F)$ form a MHS.\\
 This positive answer
  to a question of Deligne was one of the
 starting point of the linear aspect of degeneration  theory
 developed here, but the main development occurred with the discovery
 of Intersection cohomology.

 There is no such limit filtration $F$ in general  for a VMHS.

\section{Degeneration of VMHS} Families of
algebraic varieties parameterized by a non singular algebraic
curve, acquire in general singularities changing their topology at
a finite number of points of the curve. If we center a disc $D$ at
one of these points, we are in the above case of a family over
$D^*$ which extends over the origin in an algebraic family over
$D$. The fiber at the origin may be changed by modification along
a subvariety,
which do not change the family over $D^*$. \\
The study of the degeneration follows the same pattern as the
definition of the VMHS. The main results have been established
first for the degeneration of abstract VHS \cite {Sc}, then a
geometric construction   has been given in  the case of
degeneration of smooth families \cite{S}. These results will be
assumed since we concentrate our attention on the singular case
here.

The degeneration of  families of singular varieties is reduced to
the case of smooth families by the technique of simplicial
 coverings already  mentioned.
Such covering by simplicial smooth algebraic varieties with NCD
 above the origin and satisfying the descent cohomological
property, induce  a covering of the fibers over $D^*$ which is fit
to study the degeneration of the  cohomology of the fibers.

 In the case of open families, we use the fact that we can complete
algebraic varieties with a NCD at infinity $Z$, which moreover can
be supposed a relative NCD over the punctured disc.

\subsection{Diagonal degeneration of geometric  VMHS }
 The term diagonal refers to a type of construction of the
weight as diagonal with respect to a simplicial covering. The next
results
 describes the cohomological degeneration of the MHS
  of an algebraic family over a
disc. Here $\Z(b)$ will denote the MHS  on the group $(2i\pi)^{b}
\Z \subset \C $  of type $(-b,-b)$ and its tensor
 product with   a MHS on a group $V$ is denoted $V(b):=
V \otimes \Z(b)$ and called twisted MHS on $V$.\\
{\it Hypothesis}. Let $f: X \to D$ be a proper analytic morphism
defined on an analytic manifold $X$ to a complex disc $D$, $Z$  a
closed analytic subspace of $X$  and suppose the fibers of $X$ and
$Z$ over $D$ algebraic,
 then for $D$ small enough: \\
 the weight filtration  on the family $H^n( X_t- Z_t, \Q)$
define a filtration by sub-local systems $\WW$ of $R^nf_*
\Q_{|D^*}$ on $D^*$. The graded objects $Gr^{\WW}_i R^nf_*
\Q_{|D^*}$ underly a variation of Hodge
 structures (VHS) on $D^*$ defining a limit MHS at the origin
 \cite{Sc}, \cite{S}.
 The construction below gives back this limit MHS in the VHS case
 for $Z = \emptyset$ and is  deduced from this case
 by the diagonalization process for a simplicial family of varieties.\\
Let ${\widetilde X^*}:= X \times_D \widetilde D^* $ (resp.
${\widetilde Z^*} := Z \times_D \widetilde D^*$) be the inverse
image on the universal covering $ \widetilde D^*$ of $D^*$, then
the inverse image $\widetilde \WW$ of $\WW$ on $ \widetilde D^*$
is trivial and defines a filtration $W^f$ by subspaces of
$H^n({\tilde X^*}-{\widetilde Z^*},\Q)$, called here the finite
weight filtration.
 \begin{thm} There exists a MHS on the
cohomology $H^n(\widetilde X^*-\widetilde Z^*, \Z)$ with weight
filtration $W$ defined over $\Q$ and Hodge filtration defined over
$\C$ satisfying\\
i) the finite filtration $W^f$ is a filtration by sub-MHS of
$H^n(\widetilde
X^*-\widetilde Z^*,\Q)$.\\
 ii) the induced $MHS$ on $ Gr^{W^f}_i H^n (\widetilde
X^*-\widetilde Z^*, \Q)$ coincides with the limit MHS defined by
the VHS on the family $
Gr^{W}_i H^n ( X_t- Z_t, \Q)$.\\
 iii) Suppose moreover that $f$ is quasi-projective, then for all
 integers $a$ and $b$, the logarithm of the unipotent part $T^u$
  of the monodromy $N =
 \frac{1}{2i \pi} Log T^u$ induces an isomorphism for $b \geq 0$
 \[ Gr N^b: Gr^W_{a+b} Gr^{W^f}_a H^n(\widetilde X^*-\widetilde Z^*,
\Q)  \simeq  Gr^W_{a-b} Gr^{W^f}_a H^n(\widetilde X^*-\widetilde
Z^*, \Q) ( - b) \]
 \end{thm}
\begin{rem}[Category of limit MHS]
 The above MHS will be called the limit MHS of the  VMHS
 defined by the fibers of $f$. In general we  define a structure called
 limit MHS, by the following data: $(V, W^f, W, F, N)$ where $(V, W, F)$
 form a MHS, $W^f$ is an increasing filtration by sub-MHS and $N$
 is a nilpotent endomorphism of MHS: $(V, W, F) \to (V, W, F)(-1)$
 compatible with $W^f$ such that $W$ is the relative weight
 filtration of $(V, W^f, N)$.

 Limit MHS form an additive category with kernel and cokernel
 but which is not abelian since $W^f$ is not necessarily strict.
 \end{rem}
 The assertion iii) characterizes the weight $W$ as the monodromy
 filtration of $N$ relative to $W^f$, which proves its existence in the case
 of geometric variations.
 The proof will occupy this  section and is based on the reduction to the smooth case,
  via a simplicial hypercovering resolution of $X$, followed by
 a diagonalization process of the weight of each term of the
 hypercovering  [5],  as in the case of the weight in the MHS
 of a singular variety [4],[5]. \\
{\it Plan of the proof}. Precisely, let  $Y = f^{-1}(0)$ and
consider a
 smooth hypercovering $\pi: X_* \to X $ over $X$ where each term
 $X_i$ is smooth and proper over $X$, such that $Z_* := \pi^{-1}(Z)$,
 $Y_* := \pi^{-1}(Y)$ and $Z_* \cup Y_* $ are NCD in $X_*$
 with no common irreducible component in $Y_*$ and $Z_*$.
 Let $V := X-Z$,  $V_i:= X_i-Z_i$, then $\pi|V_*: V_* \to V $
 is an hypercovering. Notice that only a finite number of
 terms  $X_i$ (resp. $V_i$) of the hypercovering  are needed to compute the
 cohomology of $X$ (resp. $V$). Then
 by Thom-Whitney theorems, for $D$ small enough, $X_i$ and $Z_i$
 are topological fibre bundle over $D^*$ as well  $ Z_i^*$
 is a relative NCD in  $ X_i^*$ for a large number of indices $i$,
 and  for each $t \in D^*$, $(X_*)_t$ (resp. $(V_*)_t$)
 is an hypercovering of $X_t$ (resp. $V_t$).
  Then for each index $i$, $X_i$ and the various intersections of
  the irreducible
  components of $Z_i$, are proper and smooth families over $D^*$,
  so that  we can apply  in this situation  the results of
  J. Steenbrink
  on the degeneration process for a
 geometric family of HS \cite {S}.
The method consists first to compute  the
   hypercohomology of  the sheaf of  the nearby
cycles ${\Psi}_f (\C_V)$ on  $Y-Y \cap Z$ as the hypercohomology
of a simplicial nearby cycles ${\Psi}_{f \circ \pi}( \C_{V_*})$
 \[  \H^n(Y-Y \cap Z,{\Psi}_f (\C_V)) \simeq H^n({\widetilde V^*},
\C)  \simeq H^n({\widetilde V_*^*}, \C)\simeq \H^n(Y_*-Y_* \cap
Z_* ,{\Psi}_f( \C_{V_*})) \] where we denote by tilde the inverse
image of a space over $D^*$ to the universal cover  $\widetilde
D^*$   and where  the third term is the cohomology of the
simplicial space ${\widetilde V_*^*}$.

We restrict the construction to   the unipotent cohomology,
 denoted by an index $u$ (that is the
subspace where the action of the monodromy is unipotent) although
the theorem is true without this condition. The cohomology is
computed  as the
   hypercohomology of the simplicial variety $Y_*$
with value in some sheaf denoted  $ \Psi_{f \circ \pi}(Log Z_*)$
that we  {\it define here}. Such complex is a logarithmic version
of the nearby cycle complex of sheaves   on $Y_*$ satisfying a
simplicial cohomological mixed Hodge complex data, such that
 \[ H^n({\widetilde V^*},
\C)^u  \simeq \H^n(Y_*, \Psi_{f \circ \pi}(Log Z_*))\]
 Then
  the assertion ii) refers to the
case of  geometric VHS, which is first generalized on each $X_i$
to the non proper case and then applied to each term $X_i$. \\
With this in mind, the method of proof
 use  general results on {\it simplicial trifiltered complexes } that we
 develop now; still we need later to describe
 explicitly the complexes involved.
 \subsection{Filtered mixed Hodge complex(FMHC)}
 The proof   involves abstract results
concerning FMHC, with three filtrations $W^f, W$ and $ F $, where
$W^f$ induces on  cohomology a filtration by sub-MHS defined by
$W$ and $F$. We define first the category of complexes with three
filtrations.
 \subsubsection{}Let $\A$ be an abelian category, $F_3 \A$ the
 category of three filtered objects of $\A$ with finite
 filtrations, and $K^+F_3\A$ the category of three filtered
 complexes of objects of $\A$ bounded at left, with morphisms
 defined up to homotopy respecting the filtrations.
 \begin{defn} A morphism  $f: (K,F_1,F_2,F_3) \to (K',F'_1,F'_2,F'_3)$
  in $K^+F_3 \A$ where $F_1, F'_1$ are increasing,
   is called a quasi-isomorphism if the following
  morphisms $f^i_j$ are bi-filtered quasi - isomorphisms for all $i \geq j$
  \[ f^i_j: (F^i_1K/F^j_1K,F_2,F_3) \to (F'^i_1K/F'^j_1K,F'_2,F'_3)\]
\end{defn}
The category $D^+F_3 \A$ is obtained from $K^+F_3 \A$ by inverting
the above quasi-isomorphisms;  the objects in $D^+F_3 \A$ are
trifiltered complexes but  the group of morphisms $Hom(K,K')$ of
complexes  change, since we add to a quasi-isomorphism $f$ in
$Hom(K,K')$
 an inverse element in $Hom (K', K)$ denoted $1/f$ s.t. $f \circ (1/f) =
Id$ (resp. $ (1/f) \circ f = Id$ ), where equal to the identity
 means homotopic to the identity of $K'$ (resp. $K$). In fact
this changes completely the category since different objects, not
initially isomorphic, may become isomorphic in the new category.
\begin{defn}[Filtered  mixed Hodge complex (FMHC)] A FMHC is given
by\\
 i) A complex $K_{\Z} \in Ob D^+(\Z)$ s.t. $H^k(K_{\Z})$ is a
 $\Z-$module of finite type for all $k$.\\
  ii)   A bi-filtered complex $(K_{\Q}, W^f, W) \in Ob D^+F_2(\Q)$
  and an isomorphism  $K_{\Q} \simeq K_{\Z}\otimes \Q$  in $D^+(\Q)$
  where $W$ (resp. $W^f$) is an increasing filtration by weight
  (resp. finite weight).\\
   iii)   A tri-filtered complex $(K_{\C}, W^f, W, F) \in Ob D^+F_3(\C)$
  and an isomorphism  $(K_{\C}, W^f, W) \simeq (K_{\Q}, W^f, W)\otimes \C$
   in $D^+F_2(\C)$ where $F$  is a decreasing filtration called Hodge filtration.\\
The following axiom is satisfied: for all $j \leq i$, the
following system is a MHC
 \begin{equation*}
\begin{split}
&(W_i^f K_{\Q}/W_j^f K_{\Q} , W),\, (W_i^f K_{\C}/W_j^f K_{\C} ,
W, F)\\
& {\alpha_i^j}:(W_i^f K_{\Q}/W_j^f K_{\Q} , W)\otimes \C \simeq
 (W_i^f K_{\C}/W_j^f K_{\C} , W)
\end{split}
\end{equation*}
\end{defn}
\subsubsection{} We define as well a sheaf version as
a cohomological FMHC on a topological space $X$,
 \begin{equation*}
\begin{split}
& K_{\Z} \in Ob D^+(X, \Z),\, (K_{\Q}, W^f, W) \in Ob D^+F_2(X,
\Q)\\
& (K_{\C}, W^f, W, F) \in Ob D^+F_3(X,\C); \, {\alpha}:( K_{\Q},
W^f, W)\otimes \C \simeq
 (K_{\C}, W^f,   W)
\end{split}
\end{equation*}
  s.t. $W^f_i/W^f_j$ is a cohomological
MHC on $X$.\\
The global section functor $\Gamma$ on $X$ can be filtered derived
using acyclic tri-filtered canonical resolutions such as Godement
flabby resolutions.
\begin{lem} The derived global section functor $R\Gamma$ of a cohomological
FMHC on $X$ is a FMHC.
\end{lem}
\begin{prop} Let $(K, W^f, W, F)$ denotes a FMHC, then\\
i) The filtrations $W[n]$ and $F$ define a MHS on the cohomology
$H^n(K)$.\\
ii) The terms of the spectral sequence defined by the filtration
$W^f$ on $K$, with induced weight $W$ and Hodge $F$ filtrations
 \[  ({}_{W^f}E^{p,q}_r, W, F)= Gr^{W^f}_{-p}
   H^{p+q}(W^f_{-p+r-1}K/W^f_{-p-r}K), W, F)
   \]
   form a MHS and the differentials $d_r$ are
   morphisms of MHS  for $r \geq 1$.\\
iii) The filtration $W^f$ is a filtration by sub-MHS and we have
\[ (Gr^{W^f}_{-p} H^{p+q}(K), W[p+q], F) \simeq ({}_{W^f}E^{p,q}_{\infty},
W, F). \]
\end{prop}

The proof is in (\cite {EI}, thm 2.8). The formula for
${}_{W^f}E^{p,q}_r $ above coincides with Deligne's definition for
the spectral sequence. On this formula, the MHS on
${}_{W^f}E^{p,q}_r$ is defined as on the cohomology of any MHC. We
prove that the differential
\[ {}_{W^f}E^{p,q}_r \xrightarrow  {}_{W^f}E^{p+r,q-r+1}_r =
 Gr^{W^f}_{-p-r}H^{p+q+1}(W^f_{-p-1}K/W^f_{-p-2r}K) \]
is compatible with MHS.  It is deduced from the connection
morphism $\d$ defined by the exact sequence of complexes
\begin{equation*}
 0 \rightarrow W^f_{-p-r}K/W^f_{-p-2r}K \rightarrow
W^f_{-p+r-1}K/W^f_{-p-2r}K \rightarrow W^f_{-p + r-1}K/W^f_{-p-r}K
\rightarrow 0
\end{equation*}
Let $ \varphi: H^{p+q+1} (W^f_{-p-r}K/W^f_{-p-2r}K) \rightarrow
H^{p+q+1}(W^f_{-p-1}K/W^f_{-p-2r}K)$ denotes the morphism induced
by the natural embedding $W^f_{-p - r}K \to W^f_{-p -1}K$ and\\
 $ \pi \circ (\varphi |): W^f_{-p} H^{p+q+1} (W^f_{-p-r}K/W^f_{-p-2r}K)
 \rightarrow
 Gr^{W^f}_{-p-r}H^{p+q+1}(W^f_{-p-1}K/W^f_{-p-2r}K)$ \\ the composition of
 the restriction $\varphi |$ of $\varphi$ to $W^f_{-p}$ with the projection
 $\pi$ onto $ {}_{W^f}E^{p+r,q-r+1}_r$
 and $W^f_{-p}H^{p+q} (W^f_{-p +
r-1}K/W^f_{-p-r}K) \xrightarrow{\partial |}W^f_{-p} H^{p+q+1}
(W^f_{-p-r}K/W^f_{-p-2r}K)$ the restriction of $\d$ to the
subspace $W^f_{-p}$ then $ d_r = \pi \circ \varphi | (\circ
\partial |)$.
 Since the connection $\d$ is compatible with MHS, as
is  $\pi \circ (\varphi |) $,  so is $d_r$.\\
 The
isomorphism $ H^{p,q} ({}_{W^f}E^{*,*}_r, d_r)\to
{}_{W^f}E^{p,q}_{r +1}$ is induced by the embedding\\
$W^f_{-p+r-1}K \to W^f_{-p+r}K$, hence it is also compatible with
$W$ and $F$. We deduce that the recurrent filtrations on
$E^{p,q}_{r+1}$  induced by $W$ and $F$ on $E^{p,q}_r$ coincide
with $W$ and $F$ on ${}_{W^f}E^{p,q}_{r+1}$.
\begin{defn}[Limit mixed Hodge complex (LMHC)]
i) A  LMHC: $(K, W^f, W, F)$ in $D^+F_3 \A$ is given by the above
data i) to iii) of a FMHC,
satisfying the following:\\
1) The sub-complexes $(W^f_i K, W, F)$ are MHC for all indices
$i$.\\
2) For all $n\in \Z$, we have induced MHC
\[ (Gr^{W^f}_n K_{\Q}, W), (Gr^{W^f}_n K_{\C}, W, F), {Gr_n \alpha}:
(Gr^{W^f}_n K_{\Q}, W)\otimes \C \simeq (Gr^{W^f}_n K_{\C}, W). \]
3) The spectral sequence of $(K_{\Q}, W^f)$ degenerates at rank
$2$:\\ $E_2(K_{\Q}, W^f) \simeq E_{\infty}(K_{\Q}, W^f)$.\\
ii) A cohomological LMHC  on a space $X$ is given by a sheaf
version of the data i) to iii) s.t. $R \Gamma (X, K, W^f, W, F)$
is a LMHC.
\end{defn}
\begin{prop} Let $(K, W^f, W, F)$ denotes a LMHC, then\\
i) The filtrations $W[n]$ and $F$ define a MHS on the cohomology
$H^n(K)$, and $W^f$ induces a filtration by sub-MHS.\\
ii) The MHS deduced from i) on $ E^{p,q}_{\infty} (K, W^f) =
Gr^{W^f}_{-p} H^{p+q}(K)$ coincide with the MHS on the
   terms of the spectral sequence $ {}_{W^f}E^{p,q}_2 $ deduced
 from the MHS on the
   terms  $ {}_{W^f}E^{p,q}_1 = H^{p+q} (Gr^{W^f}_{-p}K)$.
\end{prop}
The proof, similar to the above case of FMHC, is in  (\cite {EI},
thm. 2.13).

\subsection{Diagonal  direct image of a simplicial cohomological FMHC}
  We define naturally the  direct image $ (R \pi_* K, W,
  W^f,F)$ of a simplicial cohomological FMHC $K$ on a simplicial space
  $\pi: X_* \to X$ over $X$ \cite {HIII}. The important point here is
  that the weight $W$ is in fact a diagonal sum $\delta ((\pi)_*
  W, L)$ with respect to a filtration $L$ . This operation is of the same
  nature as the mixed cone over a morphism of MHC
  where the sum of the weight in the cone
  is also diagonal.
\begin{defn} A simplicial  cohomological FMHC
 on a simplicial  (resp. simplicial strict ) space $X_*$  is given
by a complex $K_{\Z} $,
 a bi-filtered complex $(K_{\Q}, W^f, W) $,
   an isomorphism  $K_{\Q} \simeq K_{\Z}\otimes \Q$
 and a tri-filtered complex $(K_{\C}, W^f, W, F) $ on $X_*$
  with an isomorphism  $(K_{\C}, W^f, W) \simeq (K_{\Q}, W^f, W)\otimes \C$
 \\ such that
the following axiom is satisfied: the restriction $K_p$ of $K$ to
each $X_p$ is a cohomological FMHC
\end{defn}
\subsubsection{Differential graded cohomological FMHC defined by a
simplicial cohomological FMHC}
 Let $\pi: X_* \to X $
be a simplicial space over $X$. We define  $R \pi_* K $ as a
cosimplicial cohomological FMHC  by deriving first $\pi_i$ on each
space $X_i$, on which  we deduce an intermediary structure called
a differential graded cohomological FMHC on $X$ as follows. Let $
I_* ^*$ be an injective or $\pi-$acyclic  resolution of $K$ on
$X_*$, that is a resolution $I_i^{*}$ on $X_i$ varying
functorially with the index $i$, then $\pi_* I_* ^*$ is a
cosimplicial complex of abelian sheaves
 on $X$ with
 double indices  where $p$ on $\pi_* I^p_i $ is the complex degree
and $i$  the cosimplicial degree. It is endowed by the structure
of a double complex with the differential $\sum_i (-1)^i \delta_i
x^{pq}$ deduced from the cosimplicial structure, as in \cite
{HIII}.
 Such structure is known as  a
cohomological differential graded $DG^+-$complex. If we do this
operation on the various levels, rational and complex we obtain
the following structure.\\
{\it Differential graded cohomological FMHC.} A differential
graded $DG^+-$complex $C^*_*$ is a bounded below complex of graded
objects, with two degrees, the first defined by the complex and
the second  by the gradings. It is endowed with two differentials
and viewed as a double
 complex.\\
 A  differential graded cohomological MHC
  is defined by a system of a
 $DG^+-$complex, a filtered and a bifiltered complex with
 compatibility isomorphisms
\[C^*_*, (C^*_*, W), (C^*_*, W, F)\]
 s.t. for each  degree
$n$ of the grading, the component  $(C^*_n, W, F) $  is a CMHC.
\subsubsection{ The higher direct image of a  simplicial cohomological FMHC}
  It is defined by the
 simple complex associated to the double complex  $(\pi_{i,*} I_i ^*)$
  with total differential
 involving the face maps $\delta_i$ of the simplicial structure and the differentials $d_q$ on $I_q^*$:
\[  s(\pi_{q,*} I_q ^*)_{q \in \N}^n := \oplus _{p+q=n} \pi_{q,*} I^p_q;  \quad d(x^p_q) =
d_q (x^p_q) + \Sigma_i (-1)^i \delta_i x^p_q \] The filtration $L$
with respect to the second degree will be useful
\[ L^r(s(\pi_{q,*} I_q ^*)_{q \in \N}) = s(\pi_{q,*} I_q ^*)_{q\geq r} \]
 so we can  deduce
a cohomological  FMHC on $X$ by summing into a simple complex
\[ R \pi_* K := s (R \pi_{q,*} K|X_q)_{q \in \N}:=
s (\pi_{q,*} I_q ^*)_{q \in \N} \]
\begin{defn} The diagonal direct image $ (R \pi_* K, W, W^f,F)$
of $K$ on $X_*$ is defined as follows\\
i)  The weight diagonal filtration $W$ (resp. the  finite weight
diagonal filtration $W^f$~) is
\[   \delta (W,L)_n R \pi_* K:= \oplus_p W_{n+p}
\pi_{p,*} I_{p}^*, \quad   \delta (W^f,L)_n R \pi_* K:= \oplus_p
W^f_{n+p} \pi_{p,*} I_{p}^*\]
 ii) The simple  Hodge
filtration $F$ is \,
 $ F^n  R \pi_* K:= \oplus_p F^n \pi_{p,*} I_p^*$.
\end{defn}
\begin{lem} The complex $ (R \pi_* K, W^f, W, F)$ is a
 cohomological FMHC and we have
\begin{equation*}
\begin{split}
&(W^f_i/W^f_j) ( R \pi_* K, W, F)\simeq \oplus_p   (
 (W^f_{i+p}/W^f_{j+p}) \pi_{p,*} I_{p}^*,\delta (W,L), F), \\
  &(Gr^W_i (W^f_i/W^f_j)R \pi_*
 K,F)
 \simeq \oplus_p (Gr^W_{i+p} (W^f_{i+p}/W^f_{j+p})\pi_{p,*} I_p^*[-p],F).
 \end{split}
\end{equation*}
\end{lem}
 \subsection{ Construction of a Limit MHS on the unipotent  nearby cycles}
 We illustrate the above theory by an explicit construction of a LMHC on
the nearby cycles that is
  applied to define the limit MHS of a geometric  VMHS. Let $f: V
\to D$ be a quasi-projective morphism to a disc. If $D$ is small
 enough, the morphism is a topological bundle on $D^*$, hence the higher
 direct images $R^if_*\C$ underly a variation of the MHS defined
 on the cohomology of the fibers. In order to obtain at the limit
a canonical  structure not depending on the
 choice of the general fiber at a point $t\in D^*$, we introduce
 what we call here the universal fiber to define
 the {\em nearby cycle
complex of sheaves }$\Psi_f (\C)$, of which we recall the
definition in the complex analytic setting. Let $ V_0 =
f^{-1}(0)$, $V^*= V -V_0$, $p: \widetilde D^* \rightarrow D^*$
 a universal cover of the punctured unit disc $D^*$, and
 consider the following diagram
$$
\begin{array}{ccccccc}
      \widetilde {V}^* & \xrightarrow{ \widetilde{p} } &
      V^*&\xrightarrow {j} & V &\xleftarrow i & V_0 \\
      \downarrow&&\downarrow&&\downarrow &&\downarrow \\
  \widetilde D^*&\xrightarrow {p}   & D^* &  \rightarrow  &  D & \hookleftarrow&
 \{0\}
\end{array}
$$
where  $\widetilde V^*:= V^*\times_D\widetilde D^*$. For each
complex of sheaves $\FF$ of abelian groups  on $V^*$,
 the nearby cycle complex of  sheaves $\Psi_f(\FF) $ is defined as:
$$
\Psi_f(\FF):=i^{*}R \,j_*
R\,\widetilde{p}_{*}\widetilde{p}^*(\FF).
$$
Let $\widetilde D^* = \{ u =x+iy \in \C: x < 0 \}$ and the
exponential map $p (u)={\rm exp} \, u$.
 The
translation $u \to u+2i\pi$ on $\widetilde D^*$  lifts to an
action on $\widetilde V^* $, inducing an action on
$R\,\widetilde{p}_{*}\widetilde{p}^*(\FF)$ and finally a monodromy
action $T$ on $ \Psi_f(\FF)$.\\
  The method to construct the limit MHS is to explicit a structure of
mixed Hodge complex on  nearby cycles  $ R\Gamma( Y, \Psi_f(\C)) =
R\Gamma ( \widetilde V^*, \C)$. The technique  used here puts such
structure on the complex of subsheaves $\Psi_f^u(\C)$ where the
action of $T$ is unipotent, then the problem may be reduced to
this case.  In view of recent development, the existence of the
weight filtration with rational coefficients  becomes clear in the
frame of the abelian category of perverse sheaves since the weight
filtration is exactly  the monodromy filtration defined by the
nilpotent action of the logarithm of $T$ on  the perverse sheaf
$\Psi_f^u(\Q)$, up to a shift in indices. Hence we will
concentrate here on the construction of the weight filtration on
the complex $\Psi_f^u(\C)$. The construction is carried first for
a smooth morphism, then applied to each space of a smooth
simplicial covering of $V$.
 \subsection{Smooth morphism}
For a  smooth morphism $f: V \to D $, the work of Deligne
\cite{SGA7} and the smooth proper case \cite{S} suggests to
construct the limit MHS on the universal fiber $\widetilde V^*$,
however the MHS depends on the properties at infinity of the
fibers, so we need to introduce a compactification of the morphism
$f$. Then, we suppose there exists a proper  morphism called also
$f: X \to D$ with algebraic fibres  which induces the given
morphism on $V$. This will apply to a quasi-projective morphism in
which case we may suppose the morphism $f: X \to D$ projective.
Moreover, we suppose the divisor at infinity $Z = X - V$, the
special fiber $Y = f^{-1}(0)$ and their union $Z \cup Y$ normal
crossing divisors in  $X$. To study the case of the smooth
morphism on $V = X-Z$, we still cannot use the logarithmic complex
$\Omega^*_{\widetilde X^*}(Log \widetilde Z^*)$ since $\widetilde
X^*$ is analytic in nature (it is defined via the exponential
map).
 So we need to introduce a sub-complex
of sheaves of $\Omega^*_{\widetilde X^*}(Log \widetilde Z^*)$,
essentially described in \cite{SGA7}, which underly the structure
of
cohomological FMHC.\\
{\it Construction of a FMHC}. We may  start with  the following
result.
 Let $c$ be a generator  of the cohomology $H^1(D^*, \Q )$
 and denote by $\eta_{\Q} = \cup f^*( c )$ the cup product with
the inverse image
 $f^*( c ) \in H^1 (X^*, \Q)$.
 Locally at a point $y \in Y$, a neighborhood $X_y$ is a product of  discs
 and $X_y^*:= X_y - ( Y \cap X_y)$ is homotopic to a product of $n$
punctured discs, hence $H^i(X_y^*, \Q) \simeq \wedge^i H^1(X_y^*,
\Q) \simeq \wedge^i (\Q^n)$.
 The morphism $\eta_{\Q}$
 $$H^i(X_y^*, \Q)  \xrightarrow { \eta_{\Q} }  \;H^{i+1}(X_y^*, \Q )$$
is the differential of an acyclic  complex ($H^i(X_y^*, \Q)_{i\geq
0},\, d = \eta_{\Q}$). A truncation $H^{*\geq i+1}(X_y^*, \Q )$ of
this complex defines a resolution (\cite{SGA7},lecture 14,  lemma
4.18.4)
 \begin{equation*}
 H^i(\widetilde X_y^*, \Q)^u \rightarrow H^{i+1}(X_y^*, \Q ) \to
\cdots \to H^p(X_y^*, \Q) \xrightarrow { \eta_{\Q} }
\;H^{p+1}(X_y^*, \Q )\to \cdots
\end{equation*}
of the cohomology in degree $i$ of the space $\widetilde X_y^* =
X_y \times_{D^*}\widetilde D^*$ which is homotopic to a Milnor
fiber. Dually, we have an isomorphism $
 H^i({X_y}^*, \Q)/ \eta_{\Q}
H^{i-1}({X_y}^*, \Q) \simeq H^i(\widetilde X_y^*, \Q)^u$. \\
 This construction can be lifted to
the complex level, to produce in our case a cohomological FMHC on
$Y$ computing the cohomology of the space $\widetilde {X}^* = X
\times_{D^*}\widetilde{D}^*$  homotopic to a  general fiber as follows.\\
 We use first  the logarithmic complex
${\Omega}^*_X (Log (Y \cup Z))$ to compute
 $ \hbox {R j}_* \C_{X-(Y\cup Z)}$ ($j: X-(Y\cup Z) \to X$).
On the
  level of differential forms,
 $ df/f $  represents the class   $ 2 i \pi f^*(c)$, since
 $\int_{|z|=1} dz/z = 2 i \pi$, and $\wedge df/f $
  realizes the cup product as a morphism (of  degree 1)
$$ i^*_Y ({\Omega}^*_X (Log
Y \cup Z) \xrightarrow {\eta:= \wedge df/f} i^*_Y (\Omega^*_X (Log
Y \cup Z) [1 ]$$
 satisfying
${\eta}^2 = 0$ so to get a double complex. {\it By the above local
result, the simple associated complex  is quasi-isomorphic to the
sub-sheaf of unipotent  nearby cycles ${\Psi}^u_f$} (\cite{SGA7},
(\cite{EI}, thm 2.6)). \\
 Since $Y$ and $Z$
are NCD, we can write the logarithmic complex as $\Omega^*_X (Log
Y) (Log Z)$ so to introduce {\em  the weight filtration $W^Y$
(resp. $W^Z$) with respect to $Y$ (resp. $Z$) in addition to   the
weight filtration $W^{Y \cup Z}$.} The simple complex of interest
to us is realized as a sub-complex of $\Omega^*_{\widetilde
X^*}(Log \widetilde Z^*)$ generated by $\Omega^*_X (Log Y) (Log
Z)$ and the variable $u $ on $\widetilde X^*$.  It is the image of
the complex $\C [U] \otimes \Omega^*_X (Log Y) (Log
Z)$ by the embedding $I$ defined by \\
\centerline {$I(U^p):= (-1)^{p-1} (p-1)! u^{-p}, \quad d (u\otimes
1)= 1\otimes \frac{df}{f}$}

\n since  $u = Log f$ on $\widetilde X^*$. The monodromy acts as
$T(u^p \otimes \omega) = (u + 2i \pi)^p \otimes \omega$. Now if we
put $N= \frac{1}{2i \pi} Log T$, then for $p \geq 0$, $N((Log
f)^{-p}) = - p (Log f)^{-(p+1)}$(\cite{SGA7}, examples 4.6). If we
define an action $\nu$ on $\C [U] \otimes \Omega^*_X (Log Y) (Log
Z)$  by $\nu (U^p) = U^{p+1}$,  the embedding $I$ satisfy
 $I (\nu (U^p)) = N(I(U^{p+1}))$, that is the action $\nu$ corresponds to
$N$ (we may use as Kashiwara the variable $N$ instead of $U$ to
emphasize that the action of $N$ is induced by the multiplication
by the variable $N$).
\begin{rem} If we use  $\widetilde D^* = \H $ with $p(u) = e^{2i\pi u} \in
D^*$ for covering space, then  $u =\frac{1}{2i \pi} Log f$ on
$\widetilde X^*$ and $d (u\otimes 1)= 1 \otimes \frac{1}{2i \pi}
\frac{df}{f}$ while  the monodromy  is given as  $T(u^p \otimes
\omega) = (u + 1)^p \otimes \omega $. From the embedding we deduce
the differential as  $\cup f^*(c)$, and we need to define $N = Log
T$ to get for $p \geq 0$, $N((Log f)^{-p}) =  - p (Log
f)^{-(p+1)}$.
\end{rem}
 Still to get regular filtrations we need to work on a finite complex
 deduced as a quotient modulo an
acyclic sub-complex, hence we construct  the following trifiltered
complex on which the filtrations   are regular
 \[(\Psi^{u}_Y(Log Z), W^f, W, F) \]
as follows
\begin{equation*}
\begin{split}
& (\Psi^u_Y)^r(Log Z):= \oplus_{p \geq 0, q \geq 0, p+q = r}
\Omega^{p+q+1}_X (Log Y \cup Z)/W^Y_p\\
& W_i (\Psi^u_Y)^r(Log Z):= \oplus_{ p+q = r} W^{Y \cup
Z}_{i+2p+1}\Omega^{p+q+1}_X
(Log Y \cup Z)/W^Y_p,\\
& W^f_i= \oplus_{ p+q = r} W_i^Z, W^Y_i= \oplus_{ p+q = r}
W^Y_{i+p+1}, F^i = \oplus_{ p+q = r}
F^{i+p+1}\\
\end{split}
\end{equation*}
It is the simple complex associated to the  double complex
\[(\Psi^u_Y)^{p,q}(Log Z):= \Omega^{p+q+1}_X (Log Y \cup
Z)/W^Y_p, d, \eta), p \geq 0, q \geq 0 \]
 with the usual differential $d$  of forms for
fixed $p$ and the differential $\wedge df/f$ for fixed $q$, hence
the total
differential is $D \omega = d \omega + (df/f) \wedge \omega$. \\
The projection map $(\Psi^u_Y)^{p,q}(Log Z) \to
(\Psi^u_Y)^{p+1,q-1}(Log Z)$ is the action of an endomorphism  on
the term of degree  $(p+q)$ of the complex  commuting with the
differential, hence an endomorphism $\nu:\Psi^u_Y \to \Psi^u_Y$ of
the complex. The study of such complex is reduced  to the smooth
proper case applied to $X$ and the intersections of the components
of $Z$ via the residue $Res_Z$ on $Gr_*^{W^f}\Psi^u_Y(Log Z)$.
 \begin{rem}
 By construction, the differentials  are
compatible with the above  embedding. We take the quotient by
various submodules $W^Y_p$ which
 form an acyclic sub-complex, hence we have an
isomorphism $\H^* (Y,\Psi^{u}_Y(Log Z)) \simeq H^*({\widetilde
V}^*, \C)^u $ s.t. the action of $\nu$ induces $N= \frac{1}{2i
\pi} Log T$.
\end{rem}
\begin{thm} The trifiltered complex
 $(\Psi^{u}_Y(Log Z), W^f, W, F) $ is a cohomological limit mixed Hodge complex
 (LMHC) which endows the
 cohomology $H^*({\widetilde V}^*, \C)^u $ with a limit MHS such that the
the weight  filtration $W$ is  equal to the monodromy weight
filtration relative to $W^f$.
  \end{thm}
  The theorem results from the following proposition where
 $Z = \cup_{i \in I_1} Z_i$ denotes a decomposition into components
 of $Z$, $Z^J = Z_{i_1} \cap \ldots \cap Z_{i_r} $  for
  $ J =\{ i_1,\cdots, i_r
\}\subset I_1$, $Z^r := \coprod_{J \subset I_1, |J|=r} Z^J$ .
 \begin{prop} i) Let $k_Z: Y- (Z \cap Y) \to Y$, and $ j_Z: (X^*- Z^*) \to X^*$.
  There  exists a natural quasi-isomorphism
\begin{equation*}\Psi^u_Y (Log Z) \xrightarrow{\approx} R k_{Z,*} (k_Z^* \Psi^u_f
(\C))\xrightarrow{\approx}\Psi^u_f( R j_{Z,*}\C_{X^*-Z^*})
\end{equation*}
ii) The graded part for $W^f$ is expressed with the LMHC for the
various proper smooth maps $Z^{-p}  \to D$ for $p < 0$, with
singularities along the NCD: $Z^{-p} \cap Y $
\begin{equation*}
Res_{Z}: (Gr_{-p}^{W^f}(\Psi^u_Y)(Log Z), W, F) \simeq
(\Psi^u_{Z^{-p}\cap Y}[p], W[-p], F[p])
\end{equation*}
and the spectral sequence with respect to $W^f$ is given by the
limit MHS of the unipotent cohomology of
 $(\widetilde {Z}^{-p})^* = Z^{-p}\times_D
{\widetilde D}^*$    twisted by $(p)$
\begin{equation*}
\begin{split}
 &({_{W^f}}E^{p,q}_1, W,
F) \simeq \H^{p+q}(Z^{-p} \cap Y, (\Psi^u_{Z^{-p}\cap Y}[p],
W[-p], F[p]))\simeq \\
&(\H^{2p+q}((\widetilde {Z}^{-p})^*, \C)^u, W[-2p], F[p]) \simeq
(\H^{2p+q}((\widetilde {Z}^{-p})^*, \C)^u, W, F)(p)
\end{split}
\end{equation*}
iii) The endomorphism $\nu$ shift the weight by $-2$: $\nu (W_i
\Psi^u_Y) \subset W_{i-2} \Psi^u_Y$ and preserves $W^f$. It
induces an isomorphism
\begin{equation*} \nu^i: Gr^W_{a+i}Gr^{W^f}_a \Psi^u_Y(Log Z) \xrightarrow {\sim}
Gr^W_{a-i}Gr^{W^f}_a \Psi^u_Y(Log Z)
\end{equation*}
Moreover, the action of $\nu$ corresponds to the logarithm of the
monodromy on the cohomology $H^*({\widetilde V}^*, \C)^u $ \\
iv) The induced monodromy action ${\bar \nu}$ defines an
isomorphism
 \begin{equation*} {\bar \nu}^i: Gr^W_{a+i}Gr^{W^f}_a H^n( Y,\Psi^u_Y(Log Z))
  \xrightarrow {\sim} Gr^W_{a-i}Gr^{W^f}_a H^n( Y,\Psi^u_Y(Log Z))
\end{equation*}
\end{prop}
\begin{cor} The weight filtration induced by $W$ on the cohomology
$H^*({\widetilde V}^*, \C)^u $ satisfies the characteristic
property of the monodromy weight filtration relative to the weight
filtration $W^f$.
\end{cor}
The main argument  consists to deduce iv) from the corresponding
isomorphism on the complex level in iii) after a reduction to the
proper case. We remark also (\cite{EI}, prop. 3.5) that the
spectral sequence ${_{W^f}}E^{p,q}_r$ is isomorphic to the weight
spectral sequence of any fiber $X_t-Z_t$ for $t \in D^*$ and
degenerates at rank $2$.
 \begin{proof}[ Proof of the proposition in the proper smooth case (VHS)]
  The
complex $ \Psi^u_Y(Log Z)$ for $Z = \emptyset$ coincides with
Steenbrink's complex $\Psi^u_Y $ on $Y$ in $X = V$ \cite{S}. In
this case $W^f$ is trivial, $W^{Y \cup Z} = W^Y$ on $\Omega^*_X
(Log Y )$ and $(\Psi^u_Y, W, F)$ is a  MHC since  its graded
object is expressed in terms of the Hodge complexes defined by the
embedding of the various $s$ intersections of components $Y_i$ of
$Y = \cup_{i \in I} Y_i$ denoted as\\
 $a_s: Y^s := \coprod_{J \subset
I, |J|= s} Y^J \to X $  where $Y^J = Y_{i_1} \cap \ldots \cap
Y_{i_s} $  for $ J =\{ i_1,\cdots, i_s \}$. \\ The residue
\begin{equation*} (Gr^W_r \Psi^u_Y, F)\xrightarrow{Res}
\oplus_{p\geq \,sup\, (0, -r)} a_{r+2p+1,
*}(\Omega^*_{Y^{r+2p+1}}[-r-2p], F[-p-r])
\end{equation*}
defines an isomorphism with  the HC of weight $r$ at right, then
the assertion i) of the proposition for $Z = \emptyset$  reduces
to  the quasi-isomorphism
\begin{equation*}\Psi^u_Y \xrightarrow{\approx}
\Psi^u_f (\C)
\end{equation*}
 Locally, the cohomology $ H^i(\Psi^u_Y (\C)_y) $
of the stalk at $y$ is equal to the unipotent cohomology of the
universal Milnor fiber ${\widetilde X_y}^*$ of $f$ at $y$,  hence
the quasi-isomorphism  above follows by  construction of
$\Psi^u_{Y,y}$ since
\[ H^i(\Psi^u_f(\C)_y) \simeq H^i({\widetilde X_y}^*, \C)^u \simeq
H^i(\Psi^u_{Y,y})\] To prove this local isomorphism,  we use   the
spectral sequence of
 $\Psi^u_Y$ with respect to the columns
of the underlying double complex. Since the $p-$th column  is
isomorphic to $(R j_* \C / W^Y_p) [p+1]$, the  stalk at $y$ is :
$E^{-p,q}_{1, y} \simeq H^{p+q+1}(X_y^*, \C)$, for $ p\geq 0, q
\geq 0$ and $0$ otherwise, with differential $d =  \eta$; hence
 the term $E^{-p,q}_{2, y} $ is equal to $0$ for $p > 0$ and equal
to $H^{p+q}({\widetilde X_y}^*, \C)$ for $p = 0$ where
($E^{-p,q}_{1, y},\eta $) is the resolution of the cohomology of
Milnor fiber mentioned earlier, then the global isomorphism
follows
\begin{equation*}
\H^{i}( Y, \Psi^u_{ Y})\simeq \H^{i}({\widetilde X}^*, \C)^u
\end{equation*}
In the assertion ii) we use  the residue to define an isomorphism
on the first terms of the spectral sequence with respect to $W$
with the  HS  defined by $Y^i$ after a twist
 \begin{equation*}
\begin{split}
_WE^{p,q}_1= \H^{p+q}(Y, Gr^W_{-p}\Psi^u_Y, F) = \oplus_{q'\geq \,
sup \,(0, p)}H^{2p-2q'+q}(Y^{2q' - p + 1}, \C), F) (p-q')
\end{split}
\end{equation*}
 The assertion  iii) reduces to an isomorphism
\begin{equation*} Gr^W_i \Psi^u_Y \xrightarrow{\nu^i} Gr^W_{-i} \Psi^u_Y
\end{equation*}
which can be checked easily since
\begin{equation*}
\begin{split}
W_i (\Psi^u_Y)^{p,q}:= W_{i+2p+1}\Omega^{p+q+1}_X(Log Y)/W^Y_p = &
W_{i-2} (\Psi^u_Y)^{p+1,q-1} = \\
& \cdots= W_{-i} (\Psi^u_Y)^{p+i,q-i}
 \end{split}
\end{equation*}
while the two conditions  $(p+i \geq 0, p \geq 0)$ for
$W_i(\Psi^u_Y)^{p,q}$, become successively $ (p+i)-i = p \geq 0,
p+i \geq 0$ for $W_{-i}(\Psi^u_Y)^{p+i,q-i}$, hence they are
interchanged. We end the proof  in the next section. \end{proof}
\subsection{ Polarized Hodge-Lefschetz  structure}
 The first
correct proof of the assertion  iv) is given in \cite{Sa} in the
more general setting of polarized Hodge-Lefschetz modules. We
follow \cite{G-N} for an easy exposition in our case.
\subsubsection{ Hodge Lefschetz structure.} Two endomorphisms on a
finite dimensional bigraded real vector space $L = \oplus_{i,j \in
\Z} L^{i,j} $, $ l_1:L^{i,j}\to L^{i+2,j}$ and $ l_2: L^{i,j}\to
L^{i,j+2}$, define a  Lefschetz structure if they commute and if
moreover the morphisms obtained by composition
\[  l_1^i: L^{-i,j}\to
L^{i,j}, i > 0 \,\, {\rm and} \,\,  l_2^j: L^{i,-j}\to L^{i,j}, j
>0
\]
are isomorphisms.\\
It is classical to deduce from the classical representation
theory, as in hard Lefschetz theorem, that such structure
corresponds to a finite dimensional representation of the group
$SL(2, \R)\times SL(2, \R)$;  then
 a  primitive decomposition follows
 \[ L^{i,j} = \oplus_{r,s \geq 0} l_1^r l^s_2 L^{i-2r, j-2s}_0, \,\,
 {\rm where } \,
 L^{-i,-j}_0 = L^{-i,-j} \cap Ker\, l_1^{i+1} \cap Ker \,l_2^{j+1},
 i \geq 0, j \geq 0 \]
A   Lefschetz structure  is called {\it Hodge - Lefschetz
structure } if in addition $L^{i,j}$ underly  real Hodge
structures and $l_1, l_2$ are compatible with such structures.\\
A polarization of $L$ is defined by a  real bigraded bilinear form
$S: L \otimes L \to \R $ compatible with HS, s.t.
\[ S (l_i x, y) + S (x, l_i y) = 0, \, i = 1, 2 \]
It extends into a complex Hermitian form to $L \otimes \C$ such
that the induced form $S (x, Cl_1^il_2^j y)$ is symmetric positive
definite on $L^{-i,-j}_0$, where $C$ is the Weil operator
defined by the HS.\\
A differential $d: L \to L$  is a morphism compatible with H.S
satisfying:
\begin{equation*}
\begin{split}
&d:L^{i, j}\to L^{i+1, j+1}, i,j \in \Z, \, d^2 = 0, \, [d, l_i] =
0, i = 1,2,\\ & S (d x, y) = S (x, d y), x,y \in L.
\end{split}
\end{equation*}
\begin{thm}[\cite{Sa}, \cite{G-N}] Let $( L, l_1, l_2, S, d)$ be a
 bigraded
Hodge-Lefschetz structure with a differential $d$ and polarization
$S$, then the cohomology $(H^*( L, d),l_1, l_2, S)$ is a polarized
Hodge-Lefschetz structure.
\end{thm}
We assume the theorem and that we apply  to the weight spectral
sequence,
where $n$ = dim. $X$, as follows.\\
Let $K^{i,j,k}_{\C} = H^{i+j-2k+n}(Y^{2k - i + 1}, \C)(i-k)$, for
$k \geq sup \, (0 , i)$, and $K^{i,j,k}_{\C} = 0$ otherwise. Then
the residue induces an isomorphism  of $K^{i,j}_{\C} = \oplus_{k
\geq sup \, (0 , i)} K^{i,j,k}_{\C}$ with the terms of the
spectral sequence above: $ {}_WE_1^{r, q - r} \simeq K^{r,
q-n}_{\C}$. Since the special fiber $Y$ is projective, the
cup-product with an hyperplane section class  defines a morphism
$l_1 =  \smile c $ satisfying hard Lefschetz theorem on the
various smooth proper intersections $Y^i$ of the components of
$Y$,
 while $l_2$ is defined by the action of
 $N$ on $E_1$ deduced from the action of $\nu$ on the
complex $\Psi^u_Y$. The differential $d$ is defined on the terms
of the spectral sequence which are naturally polarized as
cohomology of smooth projective varieties. Then, all the
conditions to apply the above result on differential polarized
bigraded Hodge-Lefschetz structures are satisfied, so we can
deduce
\begin{cor} For  all $q, r \geq 0$ the endomorphism $N$ induces an
isomorphism\\ of HS
\[ N^r: Gr^W_{q+r}H^q({\widetilde X}^*, \Q)^u \to
Gr^W_{q-r}H^q({\widetilde X}^*, \Q)^u(-r)\]
\end{cor}
This ends the proof in the smooth proper case.
\begin{rem}[Normal crossing divisor  case]
  Let $ X = \cup_{i \in I}X_i$   be embedded as a  NCD
 with  smooth irreducible components $X_i$, in a smooth  variety $V $
 projective over the
disc $D$,  such that the fiber $Y$ at $0$ and its union with $X$
is a NCD in $V$. Then the restriction of $f$
  to the intersections $X_J = \cap_{i\in J} X_i, \, J \subset I $,
   is a NCD $ Y_J \subset X_J$, and
  the limit MHC $\Psi^u_{Y_J}$  for various  $J,  {\emptyset \not= J \subset I}$,
   form a simplicial cohomological MHC on the semi-simplicial variety $X_*$
    defined by $X$.
  In this case the finite filtration $W^f$ on the direct image, coincides
  with the increasing filtration associated by change of indices to the canonical  decreasing
  filtration $L$ on the simplicial complex, that is $W^f_i = L^{-i}$,
   so that we can apply the general theory
  to obtain a cohomological LMHC defining the LMHS on the cohomology
$H^*(\widetilde X^*, \Q)$. This is an example of the
general singular case.\\
  If we add $X_{\emptyset} = V$ to the simplicial variety $X_*$
  we obtain the cohomology  with compact support of the general
  fiber of $V-X$ which is Poincar\'e dual to the cohomology of
  $V-X$,  the complement in $V$ of the NCD. This remark explain
   the parallel (in fact dual ) between  the logarithmic complex
   case and the  simplicial case.
  \end{rem}
\begin{proof}[ Proof of the proposition in the open smooth case] We consider
the maps \\
 $(\widetilde X^* - \widetilde Z^*) \xrightarrow {\widetilde j_Z}
       \widetilde X^* \xrightarrow {\widetilde j_Y} X$, then
 the assertion i) follows from the isomorphisms
\begin{equation*}
\begin{split}
&(\Psi^u_Y (Log Z), W^f) \simeq
 i_Y^* R \widetilde j_{Y,*}(\Omega^*_{\widetilde X^*}(Log \widetilde
 Z^*), W^{\widetilde Z^*})\\
  &\simeq i_Y^* R \widetilde j_{Y,*} (R \widetilde j_{Z,*}\C_{\widetilde X^* -
  \widetilde Z^*}, \tau) \simeq
  \Psi^u_f (R j_{Z,*}\C_{X^*-Z^*}, \tau)
\end{split}
\end{equation*}
Let $I_1 \subset I$ denotes the set of indices of the components
of $Z$, $Z^i$  the union of the intersections $Z^J$ for $J \subset
I_1, |J| = i$  and $a_{Z^i}: Z^i \to X$. The assertion ii) follows
from the corresponding  bifiltered  isomorphism residue along $Z$:
 \[ (Gr^{W^Z}_i \Omega^*_{X}(Log Y)(Log (Z), W, F)
 \simeq a_{Z^i,*}(\Omega^*_{Z^i}(Log Y \cap Z^i),W[i], F[-i]).\]
  More generally, we have residue
 isomorphisms $Res_Z$ and $Res_Y$ ( \cite {EI}, 3.3.2)
\begin{equation*}
\begin{split} Gr^W_m& (W^f_b/W^f_a)(\Psi^u_Y (Log Z), F)
\xrightarrow{Res_Z} Gr^W_{m-j}(\Psi^u_{Z^j \cap Y} [-j], F[-j])\\
&\xrightarrow{Res_Y} \oplus_{j\leq m+p, j\in [a, b], p \geq 0}
(\Omega^*_{Z^j \cap Y^{m-j+2p+1}}[-m-2p], F[-m-p])
\end{split}
\end{equation*}
where $Z^i\cap Y$ is the union of $Z^J \cap Y$, so we can deduce
the structure of LMHC from the proper case.\\
The isomorphism of complexes in the assertion iii),
 can be easily checked. While the assertion  iv) for a smooth
 proper $X \to D$,
is deduced via the above $Res_Z$ from the proper smooth projective
case
 $Z^j \to D$ for various $j$
as follows. The monodromy $\nu$ induces on the spectral sequence
the isomorphism for $p \leq 0$
\begin{equation*}(Gr^W_{i+b}({_{W^f}E_1^{p,i}}),
d_1)\xrightarrow{\nu^b} (Gr^W_{i-b}({_{W^f}E_1^{p,i}}), d_1)
 \end{equation*}
 which commutes with the differential $d_1$ equal to
  a Gysin morphism alternating with respect to
 the embeddings of components of $ Z^{-p}$ into
 $ Z^{-p-1}$. Since the isomorphism
\begin{equation*}\nu^b:(Gr^W_{2p+i+b}(H^{2p+i}( Z^{-p}, \C)_{p \in \Z},
Gysin)\xrightarrow{\sim} Gr^W_{2p+i-b}(H^{2p+i}( Z^{-p}, \C)_{p
\in \Z},Gysin)
 \end{equation*}
  has been checked in the proper case $ Z^{-p}$, we
deduce then iv)
\begin{equation*}
Gr^W_{i+b} Gr_i^{W^f}H^n( Y,\Psi^u_Y(Log Z))\xrightarrow{\nu^b}
Gr^W_{i-b} Gr_i^{W^f}H^n( Y,\Psi^u_Y(Log Z))
\end{equation*}
\end{proof}
\subsection{ Quasi-projective case }
  Let $f: V \to D$ be  quasi-projective. There exists a simplicial  smooth
  hypercovering of  $ V $  of the following type. First, we consider an
  extension into a
  projective morphism $f: X \to D$ by completing with $Z = X -V$,
then we consider    a  simplicial  smooth
  hypercovering   $ \pi: X_* \rightarrow X$  with $\pi_i :=
\pi_{|X_i}$  s.t. $Z_i :=
  \pi_i^{-1} (Z)$ consists of a NCD in $ X_i$. Let $f_i := f \circ \pi_i: X_i
  \to D$; we may suppose $Y_i$  and $Y_i \cup Z_i$  NCD in
  $X_i$ so to consider the simplicial cohomological limit MHC
$(\Psi^u_{Y_*}(Log Z_*), W^f, W, F)$ and
  its direct image $R \pi_*(\Psi^u_{Y_*}(Log Z_*), W^f, W, F)$ on
$X$, then the theorem results from the following proposition (
\cite{EI}, 3.26, 3.29)
 \begin{prop} The tri-filtered complex
 \begin{equation*}
 R \pi_*(\Psi^u_{Y_*}(Log Z_*), W^f, W, F)
 \end{equation*}
  satisfy the following
 properties\\
 i) Let $k_Z: Y-(Z\cap Y) \to Y, j_Z: X^* - Z^* \to X^*$, then
  there  exists  natural quasi-isomorphisms
\begin{equation*}R\pi_* \Psi^u_{Y_*}(Log Z_*)
 \xrightarrow{\sim} R k_{Z,*} (k_Z^* \Psi^u_f
(\C))\xrightarrow{\sim}\Psi^u_f (R j_{Z,*}\C_{X^*-Z^*})
\end{equation*}
ii) The graded part for $W^f$ is expressed in terms of the
cohomological limit MHC for the various  smooth maps $(X_i- Z_i)
\to D$
\begin{equation*}
\begin{split}
Gr_{-p}^{W^f} R\pi_*(\Psi^u_{Y_*}(Log Z_*), W, F) &\simeq \oplus_i
R\pi_{i,*}(Gr_{i-p}^{W^{Z_i}}\Psi^u_{X_i}(Log Z_i)[-i], W[-i],
F) \simeq \\
&\oplus_i R \pi_{i,*}(\Psi^u_{Z_i^{i-p} \cap Y_i}[p-2i], W[-p],
F[p-i])
\end{split}
\end{equation*}
\noindent The spectral sequence with respect to $W^f$ is given by
the  twisted LMHS on the cohomology of
 $({\widetilde Z}_i^{i-p})^* = Z_i^{i-p}\times_D
{\widetilde D}^*$
\begin{equation*}
\begin{split}
 &{_{W^f}}E^{p,q}_1 (R\Gamma (Y, R\pi_*(\Psi^u_{Y_*}(Log Z_*), W^f, W,
F)) \simeq \\
&\oplus_i \H^{p+q}(Z_i^{i-p} \cap Y_i, (\Psi^u_{Z_i^{i-p} \cap
Y_i}[p-2i], W[-p], F[p-i]))\simeq \\
&\oplus_i (H^{2p+q-2i}((\widetilde Z_i^{i-p})^*, \C)^u, W, F)(p-i)
\end{split}
\end{equation*}
iii) The monodromy $\nu$ shift the weight by $-2$: $\nu (W_i
\Psi^u_Y) \subset W_{i-2} \Psi^u_Y$ and preserves $W^f$. It
induces an isomorphism
\begin{equation*} \nu^i: Gr^W_{a+i}Gr^{W^f}_a R\pi_* \Psi^u_{Y_*}(Log Z_*)
 \xrightarrow {\sim}
Gr^W_{a-i}Gr^{W^f}_a R\pi_*\Psi^u_{Y_*}(Log Z_*)
\end{equation*}
iv)  The induced iterated monodromy action ${ \nu^i}$ defines an
isomorphism
 \begin{equation*} {\nu}^i: Gr^W_{a+i}Gr^{W^f}_a H^n( Y,
R\pi_*\Psi^u_{Y_*}(Log Z_*)) \xrightarrow {\sim}
Gr^W_{a-i}Gr^{W^f}_a H^n( Y,R \pi_*(\Psi^u_{Y_*} (Log Z_*))
\end{equation*}
\end{prop}
The proof is by  reduction to the smooth proper case, namely the
various intersections $Z_i^{j}$ of the components  of the NCD
$Z_i$ in $X_i$ as in the smooth open case. The spectral sequence
is expressed as a double complex as in the case of the diagonal
direct image in general. In particular the differential $ d_1:
{_{W^f}}E^{n-i-1,i}_1 \to {_{W^f}}E^{n-i,i}_1$ is written in terms
of alternating Gysin maps  associated to $(\widetilde
Z_j^{i+j-n})^* \to (\widetilde Z_j^{i+j-n-1})^*$ and simplicial
maps $d'$ associated to $(\widetilde Z_j^{i+j-n})^* \to
(\widetilde Z_{j-1}^{i+j-n})^*$  in the double complex  (\cite
{HIII}, 8.1.19) and (\cite {EI},3.30.1) written as
$$
\begin{array}{ccccc}
      H^{2n-2j -i}((\widetilde Z_j^{i+j-n})^*, \C)^u&\xrightarrow {\d = Gysin} &
 H^{2n+2 -2j -i}((\widetilde Z_j^{i+j-n-1})^*, \C)^u  \\
      \uparrow d' &&\uparrow d' \\
H^{2n -2j -i}((\widetilde Z_{j-1}^{i+j-n+1})^*, \C)^u &
\xrightarrow{\d = Gysin} & H^{2n+ 2 -2j -i}((\widetilde
Z_{j-1}^{i+j-n-1})^*, \C)^u
\end{array}
$$
The isomorphism iii) follows from the same property on each $X_i$
while the isomorphism iv) is deduced from the smooth case above.
\subsection{Alternative construction, existence and uniqueness}
We deduce the limit structure on cohomology of a quasi projective
family from the case of a relative open NCD in a projective smooth
family. Instead of general simplicial variety the result follows
from the simplicial variety defined by this special case.
 \subsubsection{Hypothesis} Let $f: X \to D$ be a projective
 family, $i_Z: Z \to X$ a closed embedding and $i_X: X \to P_D$
 a closed embedding in a smooth family $h$ of projective spaces over
 a disc  $D$
 s.t. $ h \circ i_X = f$. By Hironaka desingularization we
 construct diagrams
 $$\begin{array}{ccccc}
    Z''_0 & \rightarrow & X''_0 & \rightarrow & P''_0 \\
     \downarrow & &\downarrow  &  &\downarrow  \\
    Z'_0 & \rightarrow & X'_0 & \rightarrow & P'_0 \\
     \downarrow & & \downarrow  &  & \downarrow  \\
     Z_0 & \rightarrow& X_0 & \rightarrow & P_0  \\
      & \searrow & \downarrow & \swarrow &   \\
      &  & 0 &  &   \\
   \end{array}
  \quad \quad \begin{array}{ccccc}
     Z'' & \rightarrow & X'' & \rightarrow & P''_D \\
     \downarrow & &\downarrow  &  &\downarrow  \\
    Z' & \rightarrow & X' & \rightarrow & P'_D \\
     \downarrow & & \downarrow  &  & \downarrow  \\
     Z & \rightarrow& X & \rightarrow & P_D  \\
      & \searrow & \downarrow & \swarrow &   \\
      &  & D &  &   \\
   \end{array}$$
first by blowing up centers over $Z$ so to obtain a smooth space
$p: P'_D \to P_D$ such that $P'_0:= p^{-1}(P_0)$,  $Z':=
p^{-1}(Z)$ and  $P'_0 \cup Z'$ are all NCD; set  $X':= p^{-1}(X)$,
then $$p|:X'-Z' \xrightarrow{\sim} X-Z, \quad p|:P'_D - Z'
\xrightarrow{\sim} P_D - Z$$
 are isomorphisms since the modifications are all over $Z$.
 Next, by blowing up centers over $X'$ we obtain a smooth space
$q: P''_D \to P'_D$ such that $X'':= q^{-1}(X')$, $P''_0:=
q^{-1}(P'_0)$, $Z'':= q^{-1}(Z')$ and  $P''_0 \cup X''$ are all
NCD, and $q|:P''_D - X'' \xrightarrow{\sim} P'_D - X'$ is an
isomorphism. For $D$ small enough, $X'', Z''$ and  $Z'$ are
relative NCD over $D^*$. Hence we deduce the diagrams
$$\begin{array}{ccccc}
    X'' - Z'' & \xrightarrow{i''_X} & P''_D - Z'' &
     \xleftarrow{j''} & P''_D - X''\\
     q_X\downarrow & &q\downarrow  &  &q\downarrow\!\wr  \\
   X' - Z' & \xrightarrow{i'_X} & P'_D - Z' & \xleftarrow{j'} & P'_D - X'\\
   \end{array} $$
Since all modifications are above $X'$, we still have an
isomorphism induced by $q$ at right.  For dim.$P_D = d$ and all
integers $i$, the morphism $q^*: H^{2d-i}_c( P''_D - Z'', \Q) \to
H^{2d-i}_c( P'_D - Z', \Q)$ is well defined on cohomology with
compact support since $q$ is proper; its Poincar\'e dual is called
the trace morphism $Tr q: H^i( P''_D - Z'', \Q) \to H^i( P'_D -
Z', \Q)$ and satisfy the relation $ Tr q\, \circ \, q^*\, = \, Id
$. Moreover, the trace morphism is defined as a morphism of
sheaves $q_* \Z_{P''_D - Z''}
 \to \Z_{P'_D - Z'}$ \cite{V2} , hence an induced trace morphism
 $(Tr q){|(X''-Z'')}: H^i( X'' - Z'', \Q) \to H^i( X' - Z', \Q)$ is well
 defined. Taking the inverse image on a universal covering
 $ \widetilde D^*$, we get a diagram of universal fibers
$$ \begin{array}{ccccc}
  (\widetilde  X'' - \widetilde Z'')^* & \xrightarrow{i''_X}
   &(\widetilde P''_D - \widetilde Z'')^* &
     \xleftarrow{j''} & (\widetilde P''_D -\widetilde  X'')^* \\
     q_X\downarrow & &q\downarrow  &  & q\downarrow\!\wr \\
  (\widetilde X' - \widetilde Z')^* & \xrightarrow{i'_X}
  &(\widetilde P'_D - \widetilde Z')^* & \xleftarrow{j'}&
  (\widetilde  P'_D -\widetilde   X')^*\\
   \end{array}$$
\begin{prop} With the  notations of the above diagram, we have
short exact sequences
\begin{equation*}
\begin{split}
0 \to H^i((\widetilde P''_D - \widetilde Z'')^*, \Q)
\xrightarrow{(i''_X)^* - Tr q}&
 H^i(\widetilde X'' - \widetilde Z'')^*, \Q)\oplus
  H^i(\widetilde P'_D - \widetilde Z')^*, \Q) \\
  & \xrightarrow{(i'_X)^* - (Tr q){|(X''- Z'')}}
 H^i(\widetilde X' - \widetilde Z')^*, \Q) \to 0\\
 \end{split}
\end{equation*}
\end{prop}
Since we have a vertical  isomorphism $q$ at right of the above
diagram, we deduce a long exact sequence of cohomology spaces
containing the sequences of the proposition; the injectivity of
$(i''_X)^* - Tr q$ and  the surjectivity of $(i'_X)^* - (Tr
q){|X''-Z''}$ are  deduced from $ Tr q \circ q^* = Id $ and $ (Tr
q){|X''- Z''} \circ q^*{|X'- Z'} = Id $, hence the long exact
sequence of cohomology deduced from the diagram splits into short
exact sequences.
\begin{cor} The cohomology $H^i((\widetilde X -
\widetilde Z)^*, \Z)$,  is isomorphic to $H^i((\widetilde X' -
\widetilde Z')^*, \Z)$ since $X-Z \simeq X'-Z'$, carry the limit
MHS isomorphic to  the cokernel of $(i''_X)^* - Tr q$ acting as a
morphism of limit MHS.
\end{cor}
 The cokernel is defined in the additive category of limit MHS.
 We remark here that the exact sequence is strict not only for the weight
 $W$, but also for $W^f$ since
 it is isomorphic to a similar exact sequence for each fiber at a
 point $t \in D^*$, where $W^f$ is identified with the weight filtration
 on the respective cohomology  groups over the fiber at $t$. Hence the sequence remains exact
 after taking the graded part $Gr^W_*Gr^{W^f}_*$ of each term.
   The left term carry a
limit MHS as the special case of the complementary of the  NCD:
$Z'' \to D$ over $D$ into the  smooth proper variety $P''_D \to D$
, while the middle term is the complementary of the intersection
of the NCD: $Z'' \to D$ over $D$ with the NCD: $X'' \to D$ over
$D$. Both cases can be treated by the above special cases without
the general theory of simplicial varieties. Hence we deduce a the
limit MHS at right as a quotient. This shows that the limit
structure is uniquely defined by the construction on NCD and
dually the logarithmic case for smooth families.
\section{Admissible variation of mixed Hodge structure}
The degeneration properties of  VMHS of geometric origin on a
punctured disc are not necessarily satisfied for general VMHS as
it has been the case for VHS with the results of Schmid. The
notion of admissible VMHS  introduced in \cite{S-Z} over a disc,
assume all the degeneration properties of the geometric case
satisfied by definition. Such definition has been extended in
\cite{K} to analytic spaces and is satisfactory for natural
operations such as the direct image by a projective morphism  of
varieties \cite{Sa}. We mention here the main local properties of
admissible VMHS over the complement of a normal crossing divisor
(DCN) proved by Kashiwara in \cite{K}.

 As an application of this concept we describe a natural MHS on the cohomology of an
  admissible VMHS. In this setting we recall the definition
  of normal
functions and we explain recent results on the algebraicity of the
zero set of normal functions to answer a question raised by
Griffiths and Green.\\
 The results apply in
general for a VMHS with quasi-unipotent local monodromy at the
points of degeneration of the NCD , however we  assume the local
monodromy unipotent, to simplify the exposition and the proofs.

 \subsection{Definition and results}  We consider  a VMHS
$(\LL, W, F)$ on the complement $X^*$ of a NCD in an analytic
manifold $X$  with  unipotent local monodromy and we denote by
$(\LL_X, \nabla)$ Deligne's canonical extension of $\LL \otimes
\OO_{X^*}$ into an analytic vector bundle  on $X$ with a flat
connection having logarithmic singularities \cite {HI}.
 The filtration by sub-local systems $W$
define a filtration by canonical extensions of $W \otimes
\OO_{X^*}$, sub-bundles of $ \LL_X$, denoted $ \WW_X$.
 The graded object $Gr_k^{\WW_X}{\LL_X}$ is the
canonical extension of $Gr_k^W\LL$ and we know that the Hodge
filtration by sub-bundles extends on $Gr_k^{\WW_X}\LL_X$ by
Scmid's result \cite{Sc}.
\begin{defn}(\cite{S-Z}, 3.13)
 A graded-polarizable variation of mixed Hodge structure
$(\LL,W,\FF_{D^*})$ over the punctured unit disc $D^*$ with local
 monodromy $T$, is called {\it pre-admissible} if\\
i) The Hodge filtration $\FF_{D^*}\subset \LL_{D^*}$ extends to a
filtration $\FF_D$ of $\LL_D$ by sub-bundles inducing for each $k$
on
$Gr_k^{\WW_D}\LL_D$, Schmid's extension of the Hodge filtration.\\
 ii) Let $W^0 := \WW_D(0)$, $F_0 := \FF_D(0)$  denote the
 filtrations
  of the fibre $L_0:= \LL_X(0)$
  at $0 \in D$, $T$ the local monodromy at $0$, $N=\log T$, then
the following coonditions are satisfied: $N F_0^p \subset F_0^{p-1}$ for all $p \in \Z$ and  the weight
filtration $M(N,W^0)$ relative to $W^0$ exists.
\end{defn}
Notice that the extension of the filtration $\FF_{D^*}$ to $\LL_D$
cannot be deduced in general from the various Schmid's extensions
to $Gr_k^{\WW_D}\LL_D$.

  We remark that the filtrations $M := M(N,W^0)$ and $F_0$ at the origin
  define a MHS:
\begin{lem}[Deligne]
The data $(L_0,M,F_0)$ defined at the origin by the pre-admissible
VMHS: $(\LL,W,\FF)$ over $D^*$ in $D$, is a MHS.

 The endomorphism
$N$ is compatible with the MHS of type $(-1,-1)$.
\end{lem}
The proof due to  Deligne is stated in the appendix to \cite{S-Z}.
The result follows from the following properties:\\
i) $(L_0, W^0, F_0, N)$ satisfy $N F_0^p \subset F_0^{p-1}$ and $N
 W_k^0 \subset  W_k^0$,\\
ii) the relative filtration $M(N, W^0)$ exists,\\
iii) for each $k$, $(Gr_k^{W^0}L_0, M, F_0)$  is a MHS.

The admissibility property in the next definition by Kashiwara
coincide over $D^*$ with the above definition   in the unipotent
case (but not the quasi-unipotent case) as proved in \cite{K}.
\begin{defn}(\cite {K},1.9) Let $X$ be  a complex  analytic space   and
$U \subset X$ a non singular open subset, complement of a closed
analytic subset. A graded polarizable
 variation of mixed Hodge structure $(\LL,W,\FF_U)$
on $U$ is called {\it admissible}  if for every analytic morphism
$f:D\to X$ on a disc  which maps $D^*$ to $U$, the inverse
$(f_{|D^*})^*(\LL,W,\FF_U)$ is a pre-admissible variation on
$D^*$.
 \end{defn}
In the case of locally unipotent admissible VMHS, Kashiwara notice
that pre-admissible VMHS in the disc are necessarily admissible.

The following criteria in \cite {K} states that  admissibility can
be tested in codimension one:
\begin{thm}(\cite {K}, 4.5.2)
 Let $X$ be  a complex  manifold,
$U \subset X$ the complement of a NCD    and  let $Z$ be a closed
analytic subset of codimension $\geq 2$ in $X$. An admissible
VMHS:
 $(\LL,W,\FF_U)$
on $U$ whose restriction to $U - Z$ is admissible in $X - Z$, is
necessarily admissible in $X$.
\end{thm}
 In particular the existence of the relative
weight filtration at a point $y\in Y_2$  follows  from its
existence at the nearby points on $Y-Y_2$. Such result is  stated
and checked locally in terms of nilpotent orbits localized at
$y$.\\
Next, we cite the following  fundamental result for admissible
variations of MHS

\begin{thm}
Let $\LL$ be an admissible graded polarized VMHS: $(\LL,W,\FF)_{X-Y}$ on the  complement    of  a
NCD: $Y$   in a complex compact smooth algebraic variety $X$,  $j: X- Y \to X$, $Z$ a sub-NCD of $Y$,  $U:= X-Z$, then for all  degrees, the cohomology groups $\H^k(U, j_{!*}\LL)$  of the intermediate extension, carry a canonical mixed Hodge structure.
\end{thm}
This result follows from M. Saito's general theory of mixed Hodge
modules \cite{Sa} but it is  obtained here directly via the
logarithmic complex. In both cases it relies heavily on the local
study of VMHS by Kashiwara in \cite{K} that is highlighted in the
next section. We use also the  purity of the intersection
cohomology in \cite{KI}, \cite{CKS}.
 The curve case is treated in \cite{S-Z}.

\subsubsection{Properties}  1)
 We describe below a logarithmic de Rham complex 
 ${\Omega}^*_X (Log Y) \otimes {\mathcal L}_X$ 
  with coefficients in
$\LL_X$ on which the weight filtration is defined  in terms of the  local study in \cite {K} 
while the Hodge filtration is easily defined and compatible with the result in \cite{KII}.\\
2) We realize the cohomology of $ U = X-Z$ as the cohomology of   
 a cohomological mixed Hodge complex (MHC): $\Omega^*(\LL, Z)$ subcomplex of ${\Omega}^*_X (Log Y) \otimes {\mathcal L}_X$ containing the intermediate extension $ j_{!*}\LL$  of $\LL$ as a sub-MHC:  $IC(X, \LL)$ such that the quotient complex define a strucyure of MHC on $i_Z^! j_{!*}\LL [1]$.\\
 3)  If the weights of $  j_{!*}\LL$ are $\geq a$, then the weights  on $\H^i(U,  j_{!*}\LL)$ are $\geq a+i$.\\
 4) The MHS is defined dually on $\H^i_c(U,  j_{!*}\LL)$  (resp. $\H^i(Z,  j_{!*}\LL)$) of weights $\leq a + i $  if  the weights of $ j_{!*}\LL \LL$ are $\leq a$. \\
 5) Let $H $ be a smooth hypersurface intersecting transversally $Y \cup Z$ such that $H \cup Y \cup Z$ is a NCD, then the Gysin isomorphism $ i_H^* \Omega^*(\LL, Z) \simeq i_H^!\Omega^*(\LL, Z)[2] $
 is defined as an isomorphism of MHC with a shift in indices.  
  
\subsection{Local study of Infinitesimal Mixed Hodge structures
after Kashiwara} The two global results above  are determined by
the study of the local properties of VMHS.  We state here  the
local version of the definitions and results in \cite{K}, but
  we skip the proofs,  as they
are technically complex, although based on invariants in linear
algebra, as
 in the general  case of VMHS.\\
 An extensive study of infinitesimal mixed Hodge structures (IMHS)
 is needed to  state and  check locally the decomposition property of
 the graded complex for the
 weight filtration of the logarithmic complex.
\subsubsection{ Infinitesimal Mixed Hodge structure (IMHS)}
  It is convenient in analysis
to consider complex MHS $(L,W,F, {\overline F})$ where we do not
need $W$ to be rational but we suppose the three filtrations
$W,F,{\overline F}$ opposed \cite {HII}. In particular a complex
HS of weight $k$ is given as $(L,F, {\overline F})$ satisfying $L
\simeq \oplus_{p+q = k} L^{pq}$, where $L^{pq} = F^p \cap
{\overline F}^q$. In the case of a MHS with underlying rational
structure on $W$ and $L$,  ${\overline F}$  on $L$ is just the
conjugate of $F$ with respect to the rational
structure.\\
To define polarization,  we recall that the conjugate space
${\overline L}$ of a complex vector space $L$, is the same group
$L$ with a different complex structure, such that the identity map
on the group $L$ defines a real linear map $
 \sigma :  L \rightarrow {\overline L}$ and the
product by scalars satisfy the relation $ \forall \lambda \in \C,
v \in L $, $\lambda \times_{{\overline L}}\sigma ( v ) := \sigma
({\overline \lambda}\times_ L v )$, then the complex structure on
${\overline L}$ is unique. A morphism $f: V \to V'$ defines a
morphism ${\overline f}: {\overline V} \to {\overline V'}$
satisfying ${\overline f}(\sigma ( v )) = \sigma (  f(v))$.
 \subsubsection{Hypothesis} We consider a complex vector space $L$ of finite
dimension, two filtrations $ F, \overline F $ of $L$ by complex
sub-vector spaces, an integer $k$ and a non-degenerate linear map
$S: L \otimes {\overline L} \to \C$ satisfying
  \[   S(x,  \sigma
 (y) ) = (-1)^k \overline { S(y, \sigma(x) )}\, {\rm for }\,  x,y \in L \quad
 {\rm and }\quad S( F^p,
\sigma({\overline F}^q)) = 0 \,\, {\rm for } \,\,  p+q > k. \]
 Let $(N_1, \ldots, N_l)$ be a set of
mutually commuting nilpotent endomorphisms of $L$ s.t.
\[S(N_j x, y) + S(x, N_j y) = 0\quad  {\rm and }\quad N_j F^p \subset F^{p-1},
\, N_j {\overline F}^p \subset {\overline F}^{p-1}.\]
  Recall that   by definition, a MHS is of weight $w$ if
   the HS on $ Gr^W_k$ is
 of weight $w+k$.
\begin{defn}[Nilpotent Orbit]
The above data is called a (polarized)  nilpotent orbit of weight
$w$, if the following
 equivalent conditions are satisfied \cite {K}\\
1) There exists a real number $c$ s.t. $( L, (e^{i \sum t_j N_j})
F, , (e^{- i \sum t_j N_j}) \overline  F)$ is a  H.S of weight
$w $  polarized by   $S$ for all $t_j > c$.\\
2) The weight filtration $W $ of $ N = \sum_j t_j N_j $ with $t_j
> 0$ for all $j$, does not depend on the various $t_j $;  $(W,
F)$ define a MHS on $L$ of weight $w$ and the bilinear form $S_k$
s.t. $ S_k (x, y) = S( x, N^k y)$ polarizes the induced HS  of
weight $w+k$ on the primitive subspace $ P_k = Ker(N^{k+1}:Gr^W_k
\to Gr^W_{-k-2})$.
\end{defn}
 Henceforth, all nilpotent orbits  are polarized.\\
  We  consider now a filtered version of the above data
    $(L; W; F; \overline  F; N_1, \ldots, N_l)$  with an
increasing filtration $W$ s.t.  $N_j W_k \subset  W_k$
 but without any given fixed bilinear product $S$.
\begin{defn} [Mixed nilpotent orbit]
The above data $(L; W, F; \overline  F; N_1, \ldots, N_l)$ is
called a mixed nilpotent orbit (graded polarized) if for each
integer $i$,\\
 $ (Gr^W_i L; F_{|}; \overline F_{|}; (N_1)_{|},
\ldots, (N_l)_{|})$, with the restricted structures, is a nilpotent
orbit  of weight $i$ for some polarization $S_i$.
\end{defn}
 This structure is called pre-infinitesimal mixed Hodge module in
 (\cite {K}, 4.2).
 A pre-admissible VMHS: $(f^*\LL,W,\FF)_{|D^*}$ on $D^*$ defines
such structure  at $0 \in D$.
\begin{defn} [IMHS](\cite {K}, 4.3)
A mixed nilpotent orbit $(L; W, F; \overline  F; N_1, \ldots,
N_l)$ is called an infinitesimal mixed Hodge structure (IMHS) if
the following conditions are satisfied:\\
 i) For each $J \subset I = \{1,
\ldots,l\}$, the monodromy filtration $ M(J)$ of $\sum_{j \in J}
N_j$ relative to $W$ exists and satisfy $N_j M_i(J) \subset
M_{i-2}(J)$ for all $ j \in J$ and $i \in \Z$.\\
ii) The filtrations $M(I), F, \overline  F $ define a graded
polarized MHS. The filtrations $W$ and $ M(J)$ are compatible with the
MHS as well the morphisms $ N_i$ are of type ($-1,-1$).
\end{defn}
IMHS are called IMHM in
 \cite {K}; Deligne remarked, the fact  that if the relative monodromy  filtration
$M(\sum_{i \in I} N_i,W)$ exists  in the case of a mixed nilpotent orbit, then it is necessarily
 the weight filtration of a MHS. \\
The following criteria  is the infinitesimal statement which
corresponds to the result that admissibility may be checked  in
codim.$1$.
\begin{thm}(\cite {K}, 4.4.1) A mixed nilpotent orbit $(L; W, F; \overline  F; N_1, \ldots,
N_l)$ is an IMHS if the monodromy  filtration of $N_j$ relative to
$W$ exists for any $j =1, \ldots,l$.
\end{thm}
 This result is not used here and it is directly satisfied in most
  applications. Its proof is embedded
in  surprisingly important properties of IMHS valuable for their own sake.

\subsubsection{Properties of IMHS} We   describe now fundamental
properties frequently needed in various constructions in mixed Hodge theory with
degenerating coefficients. \\
 We start with an important property of a relative
weight filtration, used in various proofs.
 \begin{thm}(\cite {K}, 3.2.9) Let $(L,W, N)$ be a filtered
 space with a nilpotent endomorphism with a relative monodromy
   filtration $M(N,W)$, then for each $l$, there exists a canonical
decomposition
\[  Gr^M_l L \simeq \oplus_k Gr^W_k Gr^M_l L\]
\end{thm}
In the proof, Kashiwara describes  a natural subspace of $Gr^M_l
L$ isomorphic to $Gr^W_k Gr^M_l L$ in terms of $W$ and $N$.

 In the case of an  IMHS as above,  $(Gr^M_l Gr^W_k L, F_|,
\overline F_|)$ and  $(Gr^W_k Gr^M_l L, F_|, \overline F_|)$ are
endowed with induced H.S of weight $l$; the isomorphism in
Zassenhaus lemma between the two groups is compatible with H.S. in
this case, and $( Gr^M_l L, F_|, \overline F_|)$ is a direct sum
of H.S. of weight $l$ for various $k$.  Deligne's remark  that the
relative weight filtration $M(\sum_{i\in I} N_i,W)$  is the weight
filtration of a MHS, may be deduced from this result. Another
application is the proof of
\begin{prop}(\cite {K},5.2.5) Let $(L; W, F; \overline F; N_1, \ldots, N_l)$
be an  IMHS and for $J \subset \{1,\ldots,l\}$ set $M(J) $ the
relative weight of $ N \in  C(J) = \{ \Sigma_{j\in J}t_jN_j, t_j >
0\}$. Then, for
 $J_1, J_2 \subset \{1,\ldots,l\}$, $N_1 \in C(J_1)$,
 $M(J_1 \cup J_2) $ is the
weight filtration of $N_1$ relative to $M(J_2)$.
\end{prop}
The geometric interpretation of this result in the complement of a
NCD $Y_1 \cup Y_2$ is that the degeneration to $Y_1$ followed by
the degeneration along $Y_1$ to a point $ x \in Y_1 \cap Y_2$
yields the same limit MHS as the degeneration to $Y_2$ then along
$Y_2$ as well the direct degeneration along a curve in the
complement of $Y_1 \cup Y_2$.
\subsubsection{Abelian category of IMHS} The  morphisms of
two mixed nilpotent orbit (resp. IMHS)
  $(L; W, F; \overline F; N_1, \ldots, N_l)$ and
$(L'; W', F'; \overline F'; N_1', \ldots, N_l')$    are defined to
be compatible with both, the filtrations and the nilpotent
endomorphisms.

\begin{prop} i) The category of mixed nilpotent orbits is abelian.\\
ii)  The category of  IMHS is abelian. \\
For all $J \subset I$, $M(J)$ and $W$ are filtrations by sub-MHS
of the MHS defined by $M(I)$ and $F$ and the various functors
 defined by the filtrations $W_j, Gr_j^W, M_j(J),
Gr_j^{M(J)}$,  are exact functors.
\end{prop}

We define
 a corresponding mixed nilpotent orbit structure $Hom$ on the
  vector space
 $ Hom (L,L')$
with the following classically defined filtrations:
 $Hom(W,W')$,  $Hom(F,F')$
 and  $Hom(\overline F,\overline
F')$, and the natural   endomorphisms denoted:\\ $Hom (N_1,N_1'),
\ldots, Hom (N_l,N_l')$ on $ Hom (L,L')$. Similarly a structure
called the tensor product is defined.
\begin{rem}
 Among the specific properties of the filtrations of IMHS,
  we mention the distributivity used in various proofs.
 In general three subspaces $A, \, B, \, C$ of a vector space
 do not satisfy  the following distributivity property:
 $(A+B ) \cap C = (A \cap C) + (B \cap C)$.
  A family of filtrations $F_1, \ldots, F_n$ of a
vector space $L$ is said to be distributive if for all $p, q, r
\in \Z$,  have
 \[ (F_i^p + F_j^q )\cap F^r_k  = (F_i^p \cap F^r_k) + ( F_j^q \cap
 F^r_k) \]
 In the case of an IMHS $ (W, F, N_i, i\in I)$, for   $J_1 \subset \cdots \subset
 J_k \subset I$,
the family of filtrations $\{W, F, M(J_1),\ldots, M(J_k)\}$
  is
 distributive (\cite{K}, prop.  5.2.4).
 \end{rem}
\subsection {Deligne-Hodge theory on the cohomology of a smooth
variety} We describe now a weight filtration on the logarithmic
complex with coefficients in the canonical extension of an
admissible VMHS on the complement of a NCD, based on the local
study in \cite {K}.

\subsubsection*{Hypothesis}  Let $X$ be a smooth and compact complex
algebraic variety, $Y = \cup_{i\in I} Y_i$ a NCD in $X$ with
smooth irreducible components $Y_i$, and  $(\LL, W, F)$ a
  graded polarized VMHS on $U:= X-Y$ {\it  admissible } on $X$ with
unipotent local monodromy along $Y$.

\subsubsection*{Notations} We denote by $(\LL_X, \nabla)$ Deligne's canonical
 extension of
$\LL \otimes \OO_{U}$ into an analytic vector bundle  on $X$ with
a flat connection having logarithmic singularities.
 The filtration by sub-local systems $W$ of $\LL$
define a filtration by canonical extensions denoted $ \WW_X
\subset \LL_X$, while by definition of admissibility the
filtration $\FF_U$ extends by sub-bundles  $\FF_X \subset \LL_X$.

 The aim of this section  is to deduce from the local study in \cite {K}
 the following global result

 \begin{thm}
 There exists  a weight filtration
 $W$ and a  filtration $F$ on  the logarithmic complex with coefficients in $\LL_X$
 such that the bi-filtered complex
\[
(\Omega^*_X(Log Y)\otimes \LL_X, W, F)
 \]
 underlies a structure of mixed Hodge complex and induces
  a canonical MHS on the cohomology
 groups $H^i(U, \LL)$ of $U:= X-Y$.
\end{thm}

 The filtration $F$ is  classically  deduced on the
logarithmic complex from the above bundles $\FF_X$:
 \[
 F^p  = 0 \rightarrow
F^p{\mathcal L}_X \cdots \rightarrow
 {\Omega}^i_X (Log Y) \otimes
F^{p-i}{\LL}_X \rightarrow \cdots \dots \]
 Before giving a proof,
we need  to describe the weight filtration $W$.

\subsubsection {Local  definition of the weight $W$ on the
logarithmic complex}
 For $J
\subset I$, let $ Y_J:= \cap_{i\in J} Y_i$, $ Y_J^*:= Y_J -
\cup_{i\notin J} (Y_i \cap Y_J)$ ($ Y_{\emptyset}^*:= X - Y$) and
let the various $j:Y_J^* \to X$ denote uniformly the embeddings.
Let $X(y) \simeq D^{m+l}$ be a neighborhood of a point $y$ in $Y$,
and $U(y) = X(y) \cap U \simeq {(D^*)}^m \times D^l$ where $D$ is
a complex disc, denoted with a star when the origin is deleted.
The fundamental group $\pi_1 (U(y))$ is a free abelian group
generated by $n$ elements representing  classes of closed paths
around the origin, one for each $D^* $ in the various axis with
one dimensional  coordinate $z_i$  ( the hypersurface $Y_i$ is
defined locally by the equation $z_i=0$ ). Then the local system
$\mathcal L$ corresponds to a representation of $ {\pi}_1(U(y))$
in a vector space  $L$ defined by the action of commuting
unipotent automorphisms $T_i$ for $i \in [1,m ]$ indexed by the
local components $Y_i$ of $Y$ and called monodromy action around
$Y_i$.

Classically $L$ is viewed as the fibre of $\mathcal {L}$ at the
base point of the fundamental group $\pi_1 (X^*)$, however to
represent the fibre  of  Deligne's extended bundle at $y$, we view
{\it  $L$ as the vector space    of multivalued sections of
$\mathcal {L}$} (that is the sections of the inverse of $\mathcal
{L}$  on a universal covering of $U(y)$).

\smallskip
 \noindent   The logarithm
of the unipotent monodromy $ N_i:= -\frac{1}{2 i \pi} Log
T_i  $ is a nilpotent endomorphism.
 Recall the embedding   $ L \to \LL_X(y) : v \to  \widetilde v  $
defined by the formulas
\begin{equation}
 \widetilde v (z)  = (exp ( \Sigma_{j \in J} (
log z_{j}) N_{j})). v, \quad \nabla \widetilde v = \Sigma_{j \in
J}
 \widetilde { N_{j}.v } \otimes \frac
{dz_j}{z_j}
 \end{equation}
 \noindent where a basis of $L$ is sent on a basis of
${\LL}_{X,y}$ and  the action of $N_i$ on $L$  is determined by
 the residue of the connection.\\
 {\it  Local  description of ${R j}_{*}{\mathcal L} $}.\\
 The fibre at $y$ of the complex ${\Omega}^*_X (Log Y) \otimes
{\LL}_X $  is quasi-isomorphic to a Koszul complex as follows. We
associate to $(L, N_i), i \in [1,m ]$ a strict simplicial vector
space such that for all sequences $(i.) = (i_1 < \cdots < i_p)$

\smallskip
\centerline {$L(i.) = L , \qquad N_{i_j} \colon
 L(i. - i_j) \rightarrow  L(i.)$}
 Another   notation is  $ s(L(J), N.)_{J \subset
[1, m ]} $
 where  $J$ is identified with the strictly increasing
sequence of its elements and where $ L (J) = L $.
\begin {defn}  The  simple complex defined by the
simplicial vector space above is the Koszul complex (or the
exterior algebra) defined by $(L, N_i)$ and denoted by $\Omega (L,
N.)$.
 \end {defn}

We remark that $\Omega (L, N.)$ is quasi-isomorphic to the Koszul
complex $\Omega (L, Id - T.)$ defined by $(L, Id - T_i),  i \in
[1,m ]$.

\begin {lem}  For $M \subset I$ and  $y \in Y^{*}_M$, the
above correspondence $v \mapsto \tilde {v} $, from $L$ to
${\mathcal L }_{X,y}$,  extended to $L(i_1, \ldots, i_j) \to
({\Omega}^*_X (Log Y) \otimes {\mathcal L}_X)_y $ by $v \mapsto
\tilde {v} \frac {dz_{i_1}}{z_{i_1}}\wedge \ldots \wedge  \frac
{dz_{i_j}}{z_{i_j}}$, induces quasi-isomorphisms
 \begin{equation}
({\Omega}^*_X (Log Y) \otimes {\mathcal L}_X)_y \cong
 \Omega (L,  N_j, j \in M )\cong s(L(J), N.)_{J \subset
[1, m ]}
 \end{equation}
  hence  for $j: U \to X$ we have $({R j}_* {\mathcal L})_y \cong
 \Omega (L,  N_j, j \in M) $.
\end {lem}

\n  This description of  $({R j}_* {\mathcal L })_y$ is the model
for the description of the next various perverse sheaves.\\
{\it The intermediate extension ${ j}_{!*}{\mathcal L} $.}\\
   Let $ N_J = \prod_{j \in J}  N_j$ denotes a composition
   of endomorphisms of $L$, we consider  the strict
    simplicial  sub-complex
of the de Rham logarithmic complex  defined by $N_J L := Im N_J$
in $L(J) = L$.
\begin {defn} The  simple complex defined by the above
simplicial sub-vector space  is the intersection complex of $L$
denoted by
\begin{equation}
IC (L)\colon = s(N_J L, N.)_{J \subset M}, \ N_J L \colon= N_{i_1}
N_{i_2}\ldots N_{i_j} L , J = \{i_1, \ldots, i_j\} \end{equation}
\end {defn}
 Locally the
germ of the intermediate extension $ j_{!*} {\mathcal L}$ of
${\mathcal L}$   at a point $y \in Y^*_M$  is quasi-isomorphic to
the above complex \cite {CKS}, \cite {KI}
\begin{equation}{ j}_{!*} ({\mathcal L})_y \simeq IC
(L) \simeq s(N_J L,  N.)_{J \subset M}
 \end{equation}
\subsubsection{ Definition of $N*W$ }
 Let $(L, W, N)$ denotes an
increasing filtration
  $W$ on a vector
 space $L$ with a nilpotent endomorphism $N$ compatible with $W$ s.t.
 the relative monodromy filtration $M(N,W)$ exists, then
 a new filtration $N * W$  of $L$ is defined (\cite{K},  3.4)
 by the formula
\begin{equation}
 (N*W)_k := N W_{k+1} + M_k(N,W) \cap W_k = N W_{k+1} +
M_k(N,W) \cap W_{k+1}
 \end{equation}
 where the last equality follows from (\cite{K}, Prop 3.4.1).\\
For each index  $k$, the endomorphism $ N:
 L \to L$ induces a morphism $N: W_k \to (N* W)_{k-1}$ and
 the identity $I$ on $L$ induces a morphism $I: (N* W)_{k-1} \to W_k $.
 We remark two important properties of $N*W$ ( \cite {K},
 lemma 3.4.2):\\
i) The  relative weight filtration exists and: $M(N, N*W) = M (N,
W)$.\\
ii) We have the decomposition property
\begin{equation}
 \begin{split}
 Gr^{N*W}_k &\simeq \hbox
 {Im}\, (N: Gr^W_{k+1} L \to Gr^{N*W}_k)L )\oplus
 \hbox {Ker}\,(I  : Gr^{N*W}_k L \to
  Gr^W_{k+1}L )\\
 &{Im}\, (N: Gr^W_{k+1} L \to Gr^{N*W}_k L) \simeq
 {Im}\, (N: Gr^W_{k+1} L \to Gr^W_{k+1} L).
  \end{split}
   \end{equation}
   referred to, as $W$ and
$N*W$ form a graded distinguished pair.
\subsubsection{ The filtration $W^J$ on an IMHS $L$}
 The fiber $\LL(y)$ at
a point $y \in Y^*_M$ defines an IMHS: $(L, W, F, N_i, i\in M)$;
in particular  the relative weight filtration   $M(J, W) $ of $ N
\in C(J) = \{ \Sigma_{j\in J}t_jN_j, t_j > 0\}$ exists for all
 $ J \subset M.$
A basic lemma \cite {K}, cor. 5.5.4)asserts that:
\begin{lem} $ (L, N_1*W, F, N_i, i\in M)$
and $ (M(N_j,W), F, N_i, i\in M-\{j\})$ are IMHS.
\end{lem}
In particular, an increasing filtration $W^J$  of $L$   may be defined recursively by the star
operation
\begin{equation}
 W^J :=  N_{i_1}*( \ldots (N_{i_j}*
W)\ldots ) \,\,\hbox {for} \,\, J = \{i_1, \ldots, i_j \}
\end{equation}
(denoted $\Psi_J
*W$  in
  \cite {K}, (5.8.2), see also \cite {A}.)
It describes the
  fibre of the proposed weight filtration on $\LL(y)$ for
  $y \in Y^*_J, J \subset M$. The  filtration $W^J $ does not
depend on the order of composition of the respective
transformations $N_{i_k}*$
 since in the case of an IMHS : $N_{i_p}*
 (N_{i_q}* W )= N_{i_q}* ( N_{i_p}* W) $ for all $i_p, i_q \in J$
  according to (\cite{K},
  Prop 5.5.5). The star operation has the following
  properties for all $J_1, J_2 \subset M $ and
  $J \subset K \subset M$:
  $$ M(J_1, W^{J_2} ) = M(J_1, W)^{J_2}, \quad
  M(K,  W^J) =  M(K, W).$$
\begin{defn}
 The filtration $W$ on the de Rham complex
  associated to an IMHS $(L, W, F, N_i, i \in M)$ is defined as the
 Koszul complex
  \[
  W_k(\Omega (L, N.)):= s(W^J_{k-|J|}, N.)_{J \subset M}
 \]
\end{defn}
where for each index  $i \in M - J$, the endomorphism $ N_i:
 L \to L$ induces a morphism $N_i: W^J_k \to (N_{i}* W^J)_{k-1}$
 ( careful to the same notation for  $W$   on $L$ and on $\Omega (L, N.)$).
 It is important to add the canonical inclusion
  $I:(N_{i}* W^J)_{k-1} \to W^J_k$ to the above data.
  For example, for $\vert M \vert =
2$, the data with alternating  differentials $N_i$, is written as
follows:
\[
\begin{array}{ccc}
  L & \begin{array}{c}\xrightarrow{N_1}\\
  \xleftarrow {I}\\
  \end{array} & L \\
  I \uparrow \downarrow N_2&  & I \uparrow \downarrow N_2 \\
  L &   \begin{array}{c}\xrightarrow{N_1}\\
  \xleftarrow {I}\\
  \end{array} & L\\
\end{array} \quad
\begin{array}{ccc}
  W_k & \begin{array}{c}\xrightarrow{N_1}\\
  \xleftarrow {I}\\
  \end{array} & (N_{1}* W)_{k-1} \\
  I \uparrow \downarrow N_2&  & I \uparrow \downarrow N_2 \\
 (N_{2}*W)_{k-1} &   \begin{array}{c}\xrightarrow{N_1}\\
  \xleftarrow {I}\\
  \end{array} & (N_{1}*N_{2}* W)_{k-2} \\
\end{array}
\]
\subsubsection{ Local decomposition of  $W$}
The proof of the decomposition involves a general description of
perverse sheaves in the normally crossing divisor case (\cite {K},
section 2),
  in terms of the following de Rham data $DR(L)$
   deduced from the IMHS:
\[
 DR(L):= \{L_J, W^J, F^J,  N_{K, J } = L_J \to L_K,
  I_{J, K} = L_K \to L_J \}_{K \subset J \subset M}
 \]
where for all $J \subset M$, $L_J := L$, $F_J:= F$,  $W^J$ is the
filtration defined above,  $ N_{K,J }:= \prod_{i \in J-K} N_i$ and
$I_{J, K}:= Id : L \to L$.  A set of  properties stated  by
Kashiwara (\cite{K}, 5.6)
 are satisfied. In particular :\\
\n  {\it For each  $J \subset M$, the data: $  (L_J, W^J, F_J,
\{N_j, j \in M \})$, is an IMHS} \\
 which is essentially proved in (\cite {K}( 5.8.5) and (5.8.6)). \\
 The following result is satisfied by the data and it
 is a basic step to check the structure of MHC on the logarithmic
 de Rham complex

\begin{lem}

Let $(L, F, N_i, i\in M)$ be an IMHS, and
 for $K \subset J \subset M$, let $A := J-K$, $N_A = W_i^K L
 \to W_{i- |A|}^J L$ denotes the composition of the linear
  endomorphisms $N_j, j\in A$ and $I_A : W_{i- |A|}^J L
  \hookrightarrow  W_i^K L$ the inclusion, then we have
\begin{equation*}
  Gr^{W^J}_{i- |A|}L \simeq
 \hbox {Im} ( Gr_i N_A: Gr^{W^K}_i L \to Gr^{W^J}_{i- |A|}L)
  \oplus   \hbox{ Ker}( Gr_iI_A: Gr^{W^J}_{i- |A|}L \to
  Gr^{W^K}_i L).
 \end{equation*}
 
 \end{lem}
 
The proof by induction  on the length $|A|$ of $A$  (\cite {K}
5.6.7 and lemma 5.6.5), is based at each step
 for  $a \in J-K$, on the decomposition:
 \[
 Gr^{W^{K \cup
a}}_{i-1} L \simeq \hbox
 {Im}\, (Gr_i N_a: Gr^{W^K}_i L \to Gr^{W^{K \cup a}}_{i-1})L
  \oplus
 \hbox {Ker}\,( Gr_i I_{a} :
  Gr^{W^{K \cup a}}_{i-1}L  \to Gr^{W^K}_i L).
  \]
\begin{cor}[\cite {K}, 5.6.10] 

i)  Set
 \[ P^J_k(L):= \cap_{K \subset J, K \not= J}
 \hbox {Ker}\, (I_{J,K}: Gr^{W^J}_k \to Gr^{W^K}_{k+ \vert J-K\vert}) \subset
Gr^{W^J}_k \] then $P^J_k(L)$ has pure weight $k$ with respect to
the weight filtration $M(
\sum_{j \in J} N_j, L)$.\\
ii)$P^J_k(L)$  with the action of $N_j$ for $j \in J$, is  an
infinitesimal VHS.\\
 iii ) We have: $ Gr^{W^J}_k L
\simeq \oplus_{K \subset J} N_{J-K} P^K_{k- \vert J-K\vert}(L)$.

\end{cor}

 This corollary is proved in \cite {K}; the statement (iii)
 is proved in (\cite {K}, 2.3.1).

\begin{lem}[\cite {K} prop. 2.3.1] 

The graded vector space of the filtration $W_k, k\in \Z$  on $\Omega (L, N.)$
satisfy  the decomposition  property into a direct sum of
Intersection  complexes
\begin{equation}
 Gr^W_k(\Omega (L, N.))\simeq
\oplus IC(P^J_{k-|J|}(L)[-|J|))_{J \subset M}
\end{equation}

\end{lem}

The lemma follows  from the corollary. It is the local statement
of the structure of MHC on the logarithmic de Rham complex.

\subsubsection { Global definition and  properties of the weight $W$} 

The local study
ended with  the local decomposition into Intersection complexes.
We develop now the corresponding global results.
 Taking the residue of the connection, we
define nilpotent analytic linear endomorphisms of $\LL_{Y_i}$
compatible with the filtration $\WW_{Y_i}$ by sub-analytic
bundles:
 \[ \NN_i:= Res_i (\nabla): \LL_{Y_i}  \to \LL_{Y_i}, \quad \NN_J =
 \NN_{i_1}\cdots\NN_{i_j}:\LL_{Y_J} \to \LL_{Y_J}  \]
{\it The pure Intersection complex
of a polarized  VHS}.\\
 We introduce the global Intersection complex
 $IC (X, {\mathcal L})$ as the sub-complex of
 ${\Omega}^*_X (Log Y) \otimes {\mathcal L}_X$ whose  terms in each
 degree are $\OO_X-$modules with singularities along the strata $Y_J$
 of $Y$, defined in terms of the analytic nilpotent endomorphisms
 $ \NN_i$ and $ \NN_J$ for subsets $J \subset I$ of the set $I$ of
 indices of  the components of $Y$:

\begin {defn} 
The Intersection complex
$IC (X, \LL)$ is the sub-analytic complex
 of ${\Omega}^*_X (Log Y) \otimes {\mathcal L}_X$ whose fibre at
a point $y \in Y^*_M$ is defined in terms of a set of
  coordinates $z_i, i\in M, $
defining equations of $Y_M$,  as an ${\Omega}^*_{X,y}$ sub-module,
generated by the sections $\tilde {v} \wedge_{j \in J} \frac
{dz_j}{z_j}$ for $\tilde {v} \in N_J \LL(y)$ and $J \subset M$
($N_\emptyset = Id$).
 
 \end {defn}
 This definition is independent of the
choice of coordinates; moreover the restriction of the section is
still defined in the sub-complex near $y$, since $N_J L \subset
N_{J-i}L $ for all $i \in J$.\\
For example, for $M = \{1,2\}$ at $y \in Y_M^*  $, the sections in
$({\Omega}^2_X (Log Y) \otimes {\mathcal L}_X)_y \cap IC (X,
\LL)_y$ are generated by $\tilde {v} \frac {dz_1}{z_1}\wedge \frac
{dz_2}{z_2}$ for $\tilde {v} \in N_J\LL(y)$, $\tilde {v} \frac
{dz_1}{z_1}\wedge dz_2$ for $\tilde {v} \in N_1\LL(y)$, $\tilde
{v} dz_1 \wedge \frac {dz_2}{z_2} $ for $\tilde {v} \in N_2\LL(y)$
and $\tilde {v} dz_1 \wedge dz_2 $ for $\tilde {v} \in \LL(y)$. We
deduce from the local result

\begin {lem} The intersection complex $IC (X, {\mathcal L})[n]$
shifted by $n :=$dim $X$, where $\LL$ is locally unipotent
polarized on $X-Y$, is quasi-isomorphic to the unique auto-dual
complex on $X$, intermediate extension ${ j}_{!*}{\mathcal L}[n]$
satisfying \\
$R Hom({ j}_{!*}{\mathcal L}[n],\Q_X[2n]) \simeq {
j}_{!*}{\mathcal L}[n]$.
  \end {lem}
The shift by $n$ is needed for the compatibility with the
definitions in \cite{BBD}. The next theorem is proved in \cite{KI}
and \cite{CKS}
  \begin{thm} Let $(\LL, F)$ be a polarized VHS of weight $a$, then the sub-complex
  $(IC(X,\LL),F)$ of the logarithmic
  complex with induced filtration $F$ is a Hodge complex which
   defines a pure HS of weight $a+i$ on the Intersection cohomology
$IH^i(X,\LL)$.
\end{thm}
The proof is in  terms of $L^2-$cohomology or square integrable
forms with coefficients in Deligne's extension $\LL_X$ and an
adequate metrics. The filtration $F$ on $IC(X, \LL)$ defined in an
algebro-geometric way yields the same Hodge filtration as in
$L^2-$cohomology  as proved elegantly in \cite {KII} using the
auto-duality of the Intersection cohomology.\\
{\it The global filtration  $W^J$.}
 The relative monodromy weight filtrations
$\MM (J, \WW_{Y_J}) := \MM (\sum_{i\in J}\NN_i, \WW_{Y_J}) $ of
$\sum_{i\in J}\NN_i$ with respect to the restriction $\WW_{Y_J}$
of $\WW_X$ on $\LL_X$ to $Y_J$, exist for all $J \subset I$, so
that we can define the global filtrations
\[ (\NN_i*  \WW_{Y_i})_k:= \NN_i \WW_{Y_i, k+1}+\MM_k(\NN_i,
\WW_{Y_i})\cap \WW_k = \NN_i \WW_{Y_i, k+1}+\MM_k(\NN_i,
\WW_{Y_i})\cap \WW_{k+1}\]
 and for all $J \in I$
 an increasing filtration
$\WW^J$ of $\LL_{Y_J}$ recursively by the star operation
\begin{equation*}
 \WW^J:=   \NN_{i_1 }*( \ldots (\NN_{i_j } *
\WW)\ldots ) \,\,\hbox {for} \,\, J = \{i_1, \ldots, i_j \}
\end{equation*}
\begin{defn}
The filtration $ W $ on the de Rham complex with coefficients in
the canonical extension $\LL_X$ defined by $\LL$
  is constructed by a decreasing induction on the dimension  of
  the strata $Y^*_J$ as follows:\\
  i) On $U:= X-Y$, the sub-complex  $(W_r)_{|U}$ coincides
with ${\Omega}^*_U \otimes (\WW_r)_{|U} \subset
\Omega^*_U \otimes \LL_U$.\\
 ii) We suppose $W_r$ defined on the complement of
the closure of the union of strata $\cup_{|M|=m} Y_M$, then for
each point $y \in Y^*_M$ we define
 $W_r$  locally in a neighborhood of $y \in Y^*_M$, on
 $({\Omega}^*_X (Log Y) \otimes {\LL}_X)_y$,
 in terms of the IMHS $(L, W, F, N.)$ at $y$ and a set
  of coordinates $z_i$ for $ i\in M, $
 defining a set of local equations of  $Y_M$ at $y$:\\
$ W_r $ is generated as an ${\Omega}^*_{X,y}-$ sub-module
  by the germs of the sections $\wedge_{j \in J} \frac
{dz_j}{z_j} \otimes \tilde {v} $ for $v \in W_{r-|J|}^JL$ where
$\tilde {v} $ is the corresponding germ of $(\LL_X)_y$.
\end{defn}

 The  definition of $W$ above is independent of the choice
of coordinates on a neighborhood $U(y)$, since if we choose a
different coordinate $z'_i = f z_i$ instead of $z_i$ with $f$
invertible holomorphic at $y$, we check first that the submodule
$W_{r-|J|}^J (\LL_X)_y$ of $\LL_{X,y}$ defined by the image of $
W_{r-|J|}^JL$ is independent of the coordinates as in the
construction of the canonical extension. Then we check that for a
fixed $ \alpha \in W_{r-|J|}^J \LL_{X,y} $, since the difference $
\frac {dz'_i}{z'_i} - \frac {dz_i}{z_i} = \frac {d f}{f}$ is
holomorphic at $y$, the difference of the sections $\wedge_{j \in
J} \frac {dz'_j}{z'_j}\otimes \alpha - \wedge_{j \in J} \frac
{dz_j}{z_j} \otimes \alpha $  is still a section of the
${\Omega}^*_{X,y}-$sub-module generated
  by the germs of the sections $\wedge_{j \in (J-i)} \frac
{dz_j}{z_j} \otimes W_{r-|J|}^J (\LL_X)_y $.\\
Finally, we remark that the sections defined by induction at $y$
restrict to sections already defined by induction on $(U(y)- Y^*_M
\cap U(y))$.\\
 {\it The bundles $\PP^J_k(\LL_{Y_J})$}. Given a subset $J
\subset I$, the filtration $\WW^J$ induces a filtration by
sub-analytic bundles of $\LL_{Y_J}$, then we introduce the
following analytic bundles
 \[ \PP^J_k(\LL_{Y_J}):= \cap_{K \subset J, K \not= J}
 \hbox {Ker}\, (I_{J,K}: Gr^{\WW^J}_k \LL_{Y_J} \to
 Gr^{\WW^K}_{k+ \vert J-K\vert}) \LL_{Y_J} \subset
Gr^{\WW^J}_k \LL_{Y_J} \] where $I_{J,K}$ is induced by the
natural inclusion
 $\WW^J_k \LL_{Y_J} \subset \WW^K_{k+ \vert J-K \vert}
 \LL_{Y_J}$. In particular $\PP^{\emptyset}_k(\LL_X)
 = Gr^{\WW}_k \LL_X $ and $\PP^J_k(\LL_{Y_J}) = 0$ if
 $  Y_J^* = \emptyset $.

\begin{prop} i) The weight $W[n]$ shifted by $n:=$dim.$X$  is
a filtration by perverse sheaves defined over $\Q$, sub-complexes
of $Rj_* \LL[n]$. \\
ii) The bundles $\PP^J_k(\LL_{Y_J})$ are  Deligne's extensions of
local systems $\PP^J_k(\LL)$ on $Y_J^*$. \\
 iii) The graded
perverse sheaves for the weight filtration, satisfy the
 decomposition property into intermediate extensions for all $k$
 \[
 Gr^W_k (\Omega^*_X(Log Y)\otimes \LL_X)[n] \simeq \oplus_{J \subset
 I} (i_{Y_J})_* j_{!*}\PP^J_{k-|J|}(\LL)[n-|J|].
 \]
 where $j$ denotes uniformly the inclusion of $Y_J^*$ into $Y_J$
 for each $J \subset I$,
   $\PP^J_p(\LL)$ on $Y_J^*$  is a polarized VHS
pure   with respect to the weight  induced by $\MM( \sum_{j \in J}
\NN_j, \LL_{Y_J^*})$.
\end{prop}
The proof is essentially based on the local study above which
makes sense over $\Q$ as $\LL$ and $ W $ are defined over $\Q$. In
particular, we deduce that the various graded  complexes $Gr^W_k
(\Omega^*_X(Log Y)\otimes \LL_X)[n]$ are Intersections complexes
over $\C$ from which we deduce that  the extended filtration $W_k$
on the de Rham complex satisfy the condition of support of
perverse sheaves with respect to the stratification defined by
$Y$. Similarly, the proof apply to the
Verdier dual of $W_k$ as the complexes $Gr^W_k$ are auto-dual.\\
 We need to prove that  the local rational structure
 of the complexes $W_k$ glue into a global rational structure,
 as perverse sheaves may be glued as the usual sheaves, although
 they are not concentrated
 in a unique degree.
 Since the total complex $Rj_* \LL$ is defined over $\Q$, the
 gluing isomorphisms induced on the various extended $W_k$
 are also defined over $\Q$.
 Another proof of the existence of the rational structure is based
 on Verdier's specialization \cite{EII}. The next result is
 compatible with \cite{WII}, cor 3.3.5).
\begin{cor}  The de Rham logarithmic mixed Hodge complex of
an admissible VMHS of weight $\omega \geq a$ induces on the
cohomolgy $\H^i (X-Y, \LL)$ a MHS of weight $ \omega \geq a+i$.
\end{cor}
Indeed, the weight $W_k$ on the logarithmic  complex vanishes for
$k \leq a$.
\begin{cor}  The Intersection complex $(IC(X, \LL)[n], W, F)$
of an admissible VMHS, with induced filtration as an embedded
sub-complex of the de Rham logarithmic mixed Hodge complex, is a
mixed Hodge complex satisfying for all $k$:
 \[Gr ^W_k IC(X, \LL) =
IC(X, Gr^W_k \LL).
 \]
\end{cor}
 The  existence of  relative filtrations is important
  since in general the
 intersection complex of an extension of two local systems, is not the
  extension of their intersection complex.
   We need  to check, for each
   $J \subset I$ of
length $j$, the following property of the induced filtration
   $\WW^J \cap \NN_J \LL_{Y^J}$ on $\NN_J \LL_{Y^J}$:
\[Gr_{k-j}^{\WW^J} (\NN_J\LL_{Y^J}) \simeq \hbox {Im}\,(\NN_J :
Gr^{\WW}_k \LL_{Y^J} \rightarrow Gr^{\WW}_k \LL_{Y^J}).\]
 The problem is local.
 We prove the following   statement by induction
 on the length $j$ of  $J$:
  For  each $J$ of  length $j > 0$, we have a split exact sequence:
\[
 0 \to Gr^{W^J}_{k-j} (N_JL) \to Gr^{W^J}_{k-j} L
\to  Gr^{W^J}_{k-j}(L/ N_JL)  \to 0
\]
and an isomorphism: $Gr^{W^J}_{k-j} (N_JL) \simeq
  (\hbox {Im}\,N_J : Gr^W_k L \rightarrow Gr^W_k L) $.\\
 To prove the step of the induction, we need to deduce  for all $N_i $ with
  $i \notin J$ a split exact sequence for $J \cup i$:
\[ 0 \to Gr^{W^{J\cup i}}_{k-j-1} (N_iN_JL) \to Gr^{W^{J\cup i}}_{k-j-1} L
\to  Gr^{W^{J\cup i}}_{k-j-1}(L/ N_iN_JL)  \to 0
\]
  and moreover, isomorphisms:
\begin{equation*}
\begin{split} Gr^{W^{J\cup i}}_{k-j-1} (N_iN_JL)\simeq &
  (\hbox {Im}\,N_i : Gr^{W^J}_{k-j} (N_JL) \rightarrow
  Gr^{W^J}_{k-j} (N_JL))\simeq\\
  &(\hbox {Im}\,N_iN_J : Gr^W_k L \rightarrow Gr^W_k L).
\end{split}
 \end{equation*}
 To this end we apply the following lemma:
\begin{lem}[Graded split sequence]  Let $ (L, W)$ be a filtered vector space
and $ M := M(N, W)$, then the  filtrations induced by $N*W$ on the
terms of the exact sequence : $ 0 \to NL \to L \to L/NL  \to 0 $
   satisfy the following properties:
\begin{equation}
 Gr^{N*W}_k N L  \simeq
  \hbox {Im}\,(N : Gr^W_{k+1}L \rightarrow Gr^W_{k+1}L ), \,
  Gr^{N*W}_k (L/NL) \simeq Gr^M_k (L/NL)
\end{equation}
 Moreover, the associated graded  exact sequence
\begin{equation} 0 \to Gr^{N*W}_k (NL) \to Gr^{N*W}_k L
\to  Gr^{N*W}_k(L/ NL)  \to 0
  \end{equation}
 is split with the splitting defined by the isomorphism
  \[\hbox {Ker}\,(I: Gr^{N*W}_k L
\to  Gr^W_{k+1}L)\simeq Gr^M_k (L/NL).\]
\end{lem}
The assertion is deduced from
 the   graded distinguished  pair decomposition of $Gr^{N*W}_* L$, and the
  following isomorphisms proved in (\cite{K}, cor 3.4.3): \\
$Gr^W_{k+1} Gr^{N*W}_k L  \simeq
  (\hbox {Im}\,N : Gr^W_{k+1}L \rightarrow Gr^W_{k+1}L $), and for $a \leq
  k$: $Gr^W_a Gr^{N*W}_k L \simeq$\\
  $ Gr^W_a(\hbox {Ker}\,I  :
Gr^{N*W}_k L \to  Gr^W_{k+1}L) \simeq (\hbox {Coker} \, N
   : Gr^W_a Gr^M_{k+2}L \to
   Gr^W_a Gr^M_k L )$.\\
from which we deduce the isomorphism:
 $$\hbox {Ker}\,I:
Gr^{N*W}_k L \to Gr^W_{k+1}L
 \simeq \hbox {Coker}N: W_k Gr^M_{k+2}L \to
   W_k Gr^M_k L
   \simeq W_k Gr^M_k L/NL$$
    The quotient filtration $ Q_k:= ((N*W)_k + N L) / N L$
   is isomorphic to $M_k (L/NL)$,
   since $Q_k / Q_{k-1} \simeq  Gr^M_k (L/NL)$ is
   isomorphic to $W_k Gr^M_k (L/NL) \simeq Gr^M_k (L/NL)$,
    as  the morphism $ Gr^W_{a+2} Gr^M_k L
    \xrightarrow{N} Gr^W_a Gr^M_k L$ is surjective for $a > k$.

\subsubsection {MHS on  cohomology groups of the Intersection complex}

  Let $I_1$ be a finite  subset
of $ I$  and let  $Z := \cup_{i \in I_1} Y_i$ be a sub-NCD of $Y$.
We describe next a MHS on the hypercohomology $\H^*(X-Z,
j_{!*}\LL)$. \\
Let $j': (X - Y) \to (X - Z), \,j'': (X - Y_1) \to X  $ s.t. $j =
j'' \circ j'$. The fibre at a point $y \in Z$ of the logarithmic
de Rham  complex is isomorphic to the Koszul complex
$\Omega(L,N_i, i\in M)$ for some subset $M$ of $I$. To describe
the fibre of the complex $Rj''_* (j'_{!*}\LL)_y$ as a sub-complex,
we consider $M_1:= M\cap I_1$ and $M_2:= M - M_1$, and for $J
\subset M$: $J_1 = J \cap M_1$ and $ J_2 = J \cap M_2$.

\begin {defn} With the above notations,  we define a sub-analytic complex
 ${\Omega}^*(\LL, Z)$ of the logarithmic de Rham complex
 ${\Omega}^*_X (Log Y) \otimes {\mathcal L}_X$ locally at
a point $y \in Y^*_M$  in terms of a set of
  coordinates $z_i, i\in M, $
 equations of $Y_M$:
The fiber ${\Omega}^*(\LL, Z)$ is generated
   as an ${\Omega}^*_{X,y}$ sub-module,
by the sections $\tilde {v} \wedge_{j \in J} \frac {dz_j}{z_j}$
for each  $J \subset M$ and $v \in N_{J_2} L$.
 \end {defn}

 \begin {lem} We have:
  $(Rj''_* (j'_{!*}\LL)_y \simeq ({\Omega}^*(\LL, Z)_y$
\end{lem}
 
 The intersection of a neighborhood  of $y$ with $X-Z$
 is homeomorphic to $ U =
U^*_1 \times U_2 := (D^*)^{n_1} \times D^{n_2} $. At a point $ z =
(z_1,z_2)\in U $, the fiber of the Intersection complex is
isomorphic with a Koszul complex: $(j'_{!*}\LL)_{z} \simeq IC(U_2,
L):= s(N_{J_2} L, N_i, i\in M_2 )_{J_2 \subset M_2}$ on which
$N_i, i \in M_1,$ acts. By comparison with the logarithmic de Rham
complex along $Z$, we consider the Koszul double complex: $
s(IC(U_2, L), N_i)_{i \in M_1}$ which is quasi-isomorphic to the
fibre $(Rj''_* (j'_{!*}\LL))_y $, since the terms of its classical
spectral sequence are $E_2^{p,q} \simeq R^p j''_* H^q
((j'_{!*}\LL)_y)$.
\begin{ex} In the $3-$dimensional space, let $Y_3$ be  defined by $z_3
= 0$ and $D$ a small disc at the origin of $\C$, then $(Rj''_*
(j'_{!*}\LL))_y = \R\Gamma (D^3-(D^3\cap Y_3), \LL)$ is defined by
the diagram, with  differentials defined by  $N_i$ with a + or -
sign:
\[\begin{array}{ccccc}
  L & \xrightarrow{N_1, N_2} & L \oplus L & \xrightarrow{N_1, N_2} & L\\
   \downarrow N_3&  &  \downarrow N_3&  &  \downarrow N_3 \\
 N_3L & \xrightarrow{N_1, N_2} & N_3L \oplus N_3L
 & \xrightarrow{N_1, N_2} & N_3L\\
\end{array}
\]
\end{ex}

\begin{prop} i) The filtrations $W$ and $F$ of the logarithmic
de Rham complex induce on $\Omega^*(\LL, Z)$ a structure of MHC
defining a MHS on the hypercohomology $\H^*(X-Z, j_{!*}\LL)$ such
that the graded perverse sheaves for the weight filtration,
satisfy the
 decomposition property into intermediate extensions for all $k$
 \[
 Gr^W_k \Omega^*(\LL, Z)[n] \simeq \oplus_{J_1 \subset
 I_1} (i_{Y_{J_1}})_* j_{!*} \PP^{J_1}_{k-|J_1|}(\LL)[n-|J_1|].
 \]
 where $j$ denotes uniformly the inclusion of $Y_{J_1}^*$ into $Y_{J_1}$
 for each $J_1 \subset I_1$ where $ I_1 \subset I$ and  $Z:= \cup_{i\in I_1} Y_i$.
  In particular, for $J_1 = \emptyset$,
 we have $j_{!*}Gr^W_k(\LL)[n]$, otherwise the summands are
 supported by $Z$.
\end{prop}

The proof is local and based on the properties of relative
filtrations (it is wrong otherwise). If the complex is written at
the point $y$ as a double complex
$$s(s(N_{J_2}L,N_i, i\in
M_1)_{J_1 \subset M_1})_{J_2 \subset M_2}= s(\Omega(N_{J_2}L,N_i,
i\in M_1)_{J_2 \subset M_2}$$ for each  term of  index $ J =
(J_1,J_2)$, the  filtration on $N_{J_2}L$ is induced by $
W^{(J_1,J_2)} $ on $L$, hence:
 \begin{equation*}
 Gr_{k-|J_1|-|J_2|}^{W^{(J_1,J_2)}}
N_{J_2}L \simeq N_{J_2}Gr_{k-|J_1|}^{W^{J_1}} L \simeq
 \oplus_{K \supset J_1}N_{J_2} N_{K-J_1} P^{K}_{k-|K|}(L)
  \end{equation*}
where the first isomorphism is obtained by iterating the formula
for $J$ of length $1$ in the lemma on the graded split sequence,
and the second isomorphism follows the decomposition of the second
term. Then, we can write $Gr^W_k$ of the double complex as:
$\oplus_{J_1 \subset M_1} IC(P^{K}_{k-|K|}(L)[-|J_1|]$.

\begin{rem} 1) We may always suppose that $\LL$ is a VMHS on  $X- (Y \cup Z)$ (that is to enlarge $Y$)  and consider $ Z$ as a subspace  of $Y$ equal to a union of components of $Y$.\\
2) If $Z$ is a union of intersections of components of $Y$,  these techniques
should apply to construct a sub-complex of the logarithmic de Rham
complex endowed with the structure of MHC with the induced
filtrations and hypercohomology $\H^*(X-Z, j_{!*}\LL)$; for
example, we give below the fibre of the complex at the
intersection of two lines in the plane, first when $Z = Y_1$, then
for $Z = Y_1 \cap Y_2$:
\end{rem}
\[\begin{array}{ccc}
  L & \xrightarrow{N_1} & L \\
   \downarrow N_2&  &  \downarrow N_2 \\
  N_2 L &   \xrightarrow{N_1}&   N_2 L\\
\end{array} \qquad
\begin{array}{ccc}
  L & \xrightarrow{N_1} & L \\
   \downarrow N_2&  &  \downarrow N_2 \\
  N_2 L &   \xrightarrow{N_1}& N_1 L \cap N_2 L\\
\end{array}
\]
\subsubsection{Thom-Gysin isomorphism}    Let $H $ be a smooth hypersurface intersecting transversally $Y $ such that $H \cup Y $ is a NCD, then  $ i_H^* j_{!*}\LL$ is isomorphic to the intermediate extension $ ( (j_{Y\cap H})_{!*}^* (i_H^*\LL)$   of the restriction of $\LL$ to  $H$ 
   and the residue with respect to $H$ induces an isomorphism $R_H: i_H^* ( \Omega^*(\LL, H)/ j_{!*}\LL) \simeq i_H^*j_{!*}\LL [-1]$ inverse to Thom-Gysin isomorphism $ i_H^*j_{!*}\LL  [-1] \simeq i_H^!j_{!*}\LL [1] \simeq  i_H^*(\Omega^*(\LL, H)/ j_{!*}\LL)$. \\
  Moreover, if $H $  intersects transversally $Y \cup Z$ such that $H \cup Y \cup Z$ is a NCD, then we have    a triangle
$$  (i_H)_* i_H^!\Omega^*(\LL, Z) \to  \Omega^*(\LL, Z) \to  \Omega^*(\LL, Z \cup H) \stackrel{[1]}{\rightarrow} $$
 hence the isomorphism of the quotient complex with the cohomology with support:
 $ ( \Omega^*(\LL, Z \cup H)/  \Omega^*(\LL, Z )) \simeq i_H^!\Omega^*(\LL, Z)[1] $ induced by  the connection, the   
  isomorphism of the restriction  $ i_H^* \Omega^*(\LL, Z)$ with  the complex $ \Omega^*( i_H^* \LL, Z \cap H)$ 
  constructed directly on $H$,  the residue with respect to $H$:  $\Omega^*(\LL, Z \cup H) \to 
 i_{H,*} \Omega^*( i_H^* \LL, Z \cap H)[1]$ vanishing  on $ \Omega^*(\LL, Z )$ and inducing 
  an inverse to  the Thom-Gysin isomorphism $ i_H^* \Omega^*(\LL, Z ) \simeq 
  i_H^!\Omega^*(\LL, Z)[2] $  are all  compatible with the filtrations up to  a shift in degrees.  
\subsubsection{Duality and Cohomology with compact support}

We recall first, Verdier's dual of a bifiltered complex. Let
$(K,W, F)$ be a  complex with two filtration on  a smooth compact
K\"{a}ler or complex algebraic variety $X$, and $\omega_X:=
\Q_X[2n](n)$ a dualizing complex with the trivial filtration and a
Tate twist of  the filtration by dim.$X = n$ and a degree shift by
$2n$ (so that the weight remains $0$on a complex on a smooth
compact K\"{a}ler or complex algebraic variety $X$. We denote by
$\DD (K)$ the complex  dual to $K$ with filtrations:
 \[ \DD (K):= RHom ( K,\Q_X[2n](n)), \,  W_{-i} \DD (K):=
  \DD (K/W_{i-1}), \, F^{-i}
\DD (K):= \DD (K/F^{i+1}) \]
 then we have:
$\DD Gr^W_i K \simeq Gr^W_{-i} \DD K$ and $\DD Gr_F^i K \simeq
Gr_F^{-i} \DD K$. The
 dual of a mixed Hodge complex is a MHC.\\
 In the case of $K = \Omega^*_X(Log
Y)\otimes \LL_X [n] =
 Rj_* \LL[n]$,
 the dual $\DD K = j_!\LL[n]$ with the dual structure of MHC defines the
 MHS on cohomology with compact support.
 \begin{cor} i) An admissible VMHS $\LL$ of weight $\omega \leq a$
 induces on the
cohomolgy  with compact support   $\H_c^i (X-Y, \LL)$ a MHS of
weight $ \omega \leq a+i$.\\
ii) The cohomology $\H^i (Y, j_!\LL)$ carry a MHS of weight $
\omega \leq a+i$.
\end{cor}
This result is
 compatible with \cite{WII}, thm 3.3.1).\\
 i)  The dual admissible
 VMHS $\LL^*$  of weight $\omega \geq -a$,
 and its
 de Rham logarithmic mixed Hodge complex  is of weight $\omega \geq
 -a$, with dual  quasi-isomorphic to
 $j_! \LL [2n](n)$ and
weight $\omega \leq a$. It induces on the cohomolgy $\H_c^i (X-Y,
\LL)$ a MHS of weight $ \omega \leq a+i$.\\
 ii) The weights on $\H_Y^i (X, j_{!*}\LL)$ satisfy $ \omega \geq a+i$
and by duality: $ \H^i (Y, j_{!*}\LL)$, has weights  $ \omega \leq
a+i$.
\subsubsection{The dual filtration $N!W$}
We introduce the filtration  (\cite {K}, 3.4.2)
  \begin{equation*}
 (N!W)_k := W_{k-1} + M_k(N,W) \cap N^{-1}W_{k-1}.
 \end{equation*}
 The following  morphisms are induced by $Id$ (resp. $N$) on $L$:\\
$I:  W_{k-1}  \to  (N!W)_k$ and
 $N: (N!W)_k  \to  W_{k-1}$
satisfying $N \circ I = N $ and $I \circ N = N$ on $L$. Now we
prove the duality with $N*W$.
\begin{lem}  Let $W^*$ denotes the filtration on the vector space
$L^*:= Hom(L, \Q)$ dual to a filtration $W$ on $L$, then for all
$a$
\begin{equation*}
 (N^*!W^*)_a = (N*W)^*_a \subset L^*.
 \end{equation*}
\end{lem}
Let  $M_i:= M_i(N,W)$, it is auto-dual as a filtration of the
vector space $L$: $ M^*_i= M_i(N^*,W^*)$. To prove the inclusion
of the left term into the right term for $a = -k$, we must write
an element $\varphi \in L^*$ vanishing on $ N W_k + M_{k-1} \cap
W_{k-1} = N W_k + M_{k-1} \cap W_k$ as a sum of elements $\varphi
= \gamma + \delta $ where $ \delta \in W^*_{-k-1}$ vanish on $W_k$
and $\gamma \in M^*_{-k} \cap (N^*)^{-1}W^*_{-k-1}$ vanish on $N
W_k + M_{k-1}$. We construct $\gamma$ such that  $\gamma( N W_k +
M_{k-1})= 0$ and $\gamma_{\vert W_k} = \varphi_{\vert W_k}$  which
is possible as $\varphi (M_{k-1} \cap W_k) = 0$; then we put
$\delta = \varphi - \gamma$. The opposite inclusion is clear.\\
Now we deduce from the decomposition orthogonal to the case of
$N*W$:
\begin{cor}
 We have the decomposition:
 \[
 Gr^{N!W}_k \simeq \hbox
 {Im}\, (I: Gr^W_{k-1} \to Gr^{N!W}_k) \oplus
 (\hbox {Ker}\, (N : Gr^{N!W}_k  \to
  Gr^W_{k-1}).
  \]
\end{cor}
{\it Properties of the  iterated filtrations ${\overline W}$.}\\
 We associate to each IMHS with nilpotent
  endomorphisms $N_i$ for $i \in M$ and to each subset $J \subset
  M$,  an increasing filtration ${\overline W}^J$ of $L$
recursively by the $!$ operation
\begin{equation*}
 {\overline W}^J :=  N_{i_1 }!( \ldots (N_{i_j }!
W)\ldots ) \,\,\hbox {for} \,\, J = \{i_1, \ldots, i_j \}
\end{equation*}
since for all $i$ and $j$, $N_i!N_j W = N_j!N_i! W$. This family
satisfy the data in (\cite {K}, prop. 2.3.1)
determined by the following  morphisms,  defined  for $K \subset J$:\\
$I_{K,J}: {\overline W}^K_k  \to
 {\overline W}^J_{k +  \vert J -K\vert }$ and
 $N_{J,K}: {\overline W}^J_k  \to
 {\overline W}^K_{k -  \vert J - K\vert }$,
so we deduce
\begin{lem} i)  Set for each $J \subset M$,
 \[ Q^J_k(L):= \cap_{K \subset J, K \not= J}
 \hbox {Ker}\, (N_{J,K}: Gr^{{\overline W}^J}_k \to
 Gr^{{\overline W}^K}_{k - \vert J - K \vert}),
   \]
    then $Q^J_k(L)$ has pure weight with respect to the
weight $M(
\sum_{j \in J} N_j, L)$.\\
ii) We have: $ Gr^{{\overline W}^J}_k \simeq \oplus_{K \subset J}
I_{K,J} (Q^K_{k - \vert J-K\vert}(L))$.\\
iii) Let $(L^*, W^*)$ be dual to $(L,W)$, then $Q^J_k(L^*) \simeq
(P^J_{-k}(L))^*$.
\end{lem}
\begin{rem}  By local duality at a point $x \in Y$, $\DD_x i_x^* Rj_* \LL
\simeq i_x^! \DD Rj_* \LL \simeq i_x^! j_! \DD \LL$.
 Hence the filtration ${\overline {W^*}}$ is on the
 cohomology of $j_!  \LL^*$ at the point $x$.
\end{rem}
\subsubsection {Complementary results}
We consider from now on a pure Hodge complex  of weight $\omega
(K) = a$;
 its dual $\DD K$ is pure of weight $\omega (\DD K) = - a$.
 For a pure polarized VHS $\LL$ on $X-Y$
of weight $\omega = b$, hence $\LL[n]$ of weight $\omega = a = b+n
$, the polarization defines an isomorphism $\LL (a)[n] \simeq \R
Hom (\LL[n], \Q_{X-Y}(n)[2n])\simeq \LL^*[n]$, then Verdier's
auto-duality of the intersection complex reads as follows:
\begin{equation*}
\begin{split}
&j_{!*}\LL[n] \simeq (\DD(j_{!*}\LL[n]) (-a)\\
 &\H^{-i}(X,j_{!*}\LL[n])^* = H^i(Hom(\R
\Gamma(X,j_{!*}\LL[n]),
\C) \simeq  H^i(X, \DD(j_{!*}\LL[n])\simeq \\
&H^i(X, j_{!*}\DD(\LL[n])
 \simeq H^i(X, j_{!*}(\LL^*[n])\simeq H^i(X, j_{!*}(\LL[n])(a)
\end{split}
 \end{equation*}
For $K = \LL[n]$, the exact sequence :$ i^!j_{!*} K \to j_{!*}K
\to j_*K $ yields an isomorphism $ i^!j_{!*}K[1] \simeq j_* K/
j_{!*}K$, hence
$$\DD (i^* j_*K/i^* j_{!*}K ) \simeq \DD(i^!j_{!*}K[1]) \simeq
i^*  j_{!*} \DD K [-1]$$
  For $K $ pure of weight $\omega (K) = a$, $W_a j_*K =
 j_{!*}K$, $\omega (\DD K) = -a$,
  and  for $r > a$
  $$ W_{-r} i^*  j_{!*} \DD K [-1]
 \simeq \DD  (i^* j_*K/W_{r-1}).$$
  We deduce from the polarization:
   $\DD K \simeq K(a)$ where (a) drops the weight by $-2 a$,
    a definition of the weight on $i^* j_{!*}K$ for $r > a$
  \begin{equation*}
 W_{-r+2a} i^* j_{!*}K \simeq
   \DD  (i^* j_*K/W_{r-1})[1],
   \quad GR^W_{-r+2a}i^*  j_{!*}K \simeq
    \DD Gr^W_r (i^* j_*K)[1].
 \end{equation*}
\begin{cor} i) The  complex $(i^*  j_{!*} K [-1]
, F)$ with restricted Hodge filtration and the weight filtration
for $ k > 0$
 \[  W_{a-k}i^* j_{!*} K [-1]
\simeq \DD  (i^* j_*K/W_{a+k-1}),\]
 satisfying $ Gr^W_{a-k}i^* j_{!*} K [-1]\simeq
 \DD Gr^W_{a+k} (i^* j_*K/i^* j_{!*}K)$ has the structure of a MHC.\\
(In particular, the cohomology $\H^i(Y,  j_{!*} K)$ carry a MHS of
weight $\omega \leq a +i$).
 \end{cor}
We remark: $ W_{a-1} i^* j_{!*} K [-1]\simeq i^* j_{!*} K [-1] =
\DD (i^* j_*K/  j_{!*} K) $, hence: \\
 $\H^i(Y,  j_{!*} K)=
\H^{i+1}(Y, j_{!*} K [-1])$ has weight $\omega \leq a-1 + i+1$.\\
 By the
above corollary, we have
\begin{equation*}
\begin{split}
&GR^W_{-r+2a}i^*  j_{!*}K \simeq
    \DD
 (Gr^W_r (\Omega^*_X(Log Y)\otimes \LL_X)[n])[1] \simeq\\
 & \oplus_{J \subset
 I} j_{!*}(\PP^J_{r-|J|}(\LL ))^*[n+1-|J|].
\end{split}
\end{equation*}

\begin{ex} For a polarized unipotent VHS on a punctured
disc $D^*$,  $i^* j_{!*}K [-1]$ is the complex $(L \xrightarrow
{N} NL)$ in degree $0$ and $1$, quasi-isomorphic to $ \hbox {ker}N
$, while $ (i^* j_*K/  j_{!*} K \simeq L/NL $ and $\DD (i^* j_*K/
j_{!*} K)\simeq (L/NL)^* $. The  isomorphism is induced by the
 polarization $Q$ as follows
 \[
\begin{array}{ccccc}
  (L/NL)^* &\simeq& (L^* &\xrightarrow{N^*} &(NL)^*)\\
  \uparrow \simeq&&\uparrow \simeq&& \uparrow \simeq\\
   \hbox {ker}N &\simeq& ( L
  & \xrightarrow{-N}&NL \\
\end{array}
\]
where $-N$ corresponds to $N^*$ since  $Q(Na,b) + Q (a, Nb) = 0$.
The isomorphism we use is $
  \hbox {ker}N \simeq (L/NL)^*$ in degree $0$; we set
  for $k > 0 $ :\\
 $M(N)_{a-k+1}(\hbox {ker} N) := W_{a-k}(L \to NL)\simeq
  \DD  (i^* j_*K/W_{a+k-1})
  = (L/(NL+M_{a+k-2}))^*$.
\end{ex}
\subsubsection {Final remark on the decomposition theorem} We do not cover here
the ultimate results in Hodge  theory with coefficients in an IMHS
(see \cite{BBD} and \cite{Sa}).

 We explain here basic points in the theory.
 In the presence of a projective morphism $f: X \to V$ and a
polarized VHS of weight $a$ on the smooth open subset $X-Y$ with
intermediate extension $j_{!*}\LL$ on $X$,  a theory of perverse
filtration ${^p\tau}_i(Rf_*j_{!*}\LL)$ on the direct image on $V$
has been developed in \cite{BBD}. It defines on the cohomology of
the inverse of an open algebraic set $U \subset V$, a perverse
filtration ${^p\tau}_i\H^r( f^{-1}(U), j_{!*}\LL)$. \\
 We suppose $X$ smooth, $Y$ a NCD and we consider open sets complement
  to some NCD $Z$ sub-divisor of $Y$, since
 we may reduce the problem to such case  by
 desingularizing $X$.\\
1) If $X_U := f^{-1}(U)$ is the complement of  a sub-NCD of $Y$,
 the theory asserts that
the subspaces ${^p\tau}_i$ are sub-MHS of $\H^r( f^{-1}(U), j_{!*}\LL)$.\\
2) Let $v \in V$ be a  point in the zero-dimensional strata of a
Thom-Whitney stratification compatible with $f$ and
$Rf_*j_{!*}\LL$ and suppose $X_v := f^{-1}(v)$  a sub-NCD of $Y$,
then the local-purity theorem in \cite{DG} corresponds to the
following semi- purity property of the weights $\omega$ on the
cohomology of a tubular neighborhood $T_v $ of $f^{-1}(v)$,
inverse image $
f^{-1}(B_v)$ of a neighborhood $B_v $ of $v$:\\
 i) $\omega >  a+r $ on ${^p\tau}_{\leq r} \H^r(T_v-X_v,j_{!*}\LL) $, and
 dually\\
 ii)  $\omega
\leq a+r $ on $ \H^r(T_v-X_v,j_{!*}\LL) / {^p\tau}_{\leq r}
\H^r(T_v -X_v,j_{!*}\LL) $.\\
3) The corresponding decomposition theorem on $B_v -v$ states the
isomorphism with the cohomology of a perverse cohomology:
\begin{equation*} Gr^{^p\tau}_i \H^r(T_v-X_v,j_{!*}\LL)\simeq
\H^{r-i}(B_v-v,{^p\HH}^i(Rf_*j_{!*}\LL)).
 \end{equation*}
 4) By
 iterating the cup-product with he class $\eta$ of a
  relative hyperplane section, we have Lefschetz type
 isomorphisms of perverse cohomology sheaves
\[ {^p\HH}^{-i}(Rf_*j_{!*}\LL[n])
\xrightarrow{\eta^i}{^p\HH}^i(Rf_*j_{!*}\LL[n]).
\]
In \cite{Sa}, the proof is carried via the techniques of
differential modules and is based on extensive use of Hodge theory
on the
perverse sheaves of near-by and vanishing cycles.
 A direct proof  may be obtained by
 induction on the strata on $V$.
 If we suppose the decomposition theorem   and Lefschetz
 types isomorphisms on $V- v$,  one may prove directly that
  the perverse filtration on the  cohomology $
 \H^r(T_v-X_v,j_{!*}\LL)$
  is  compatible with MHS, following the use of the monodromy filtration
 $W(\eta)$  of the nilpotent endomorphism defined by cup-product with
   the class $\eta$ of a relative hyperplane section of type $(1,1)$
 on the total cohomology $\oplus_i
 \H^r(T_v-X_v,j_{!*}\LL)$.
 This filtration $W(\eta)$ is  compatibile with  MHS and the
  graded part
 $Gr^{W(\eta)}_i$  coincide with $Gr^{^p\!\tau}_{-i}$. Then it
 makes sense to prove the semi-purity property, from which we
 deduce the extension of the decomposition over the point $v$ and
 check the Lefschetz isomorphism, which complete the inductive
 step ( for a general strata we intersect with a transversal section
 so to reduce to the case of a zero dimensional strata).

\end{document}